\documentclass{article}
\usepackage{amsmath,amsthm}
  \usepackage{epsfig} 
\usepackage{graphicx}  \usepackage{epstopdf}
 \usepackage[colorlinks=false,hidelinks]{hyperref}
\hypersetup{urlcolor=blue, citecolor=red}

\usepackage[a4paper]{geometry}

\newtheorem{theorem}{Theorem}[section]

\theoremstyle{definition}

\newcommand{\eps}[1]{{#1}_{\varepsilon}}

\newtheorem{theo}{Theorem}[section]
\newtheorem{rem}{Remark}

\newtheorem{lem}[theorem]{Lemma}
\newtheorem{propo}{Proposition}
\newtheorem{defi}[theorem]{Definition}
\newtheorem{cor}{Corollary}

\usepackage{amssymb,amsfonts}
\newcommand{\WT}{\widetilde}
\newcommand{\OV}{\overline}
\newcommand{\DST}{\displaystyle}
\newcommand{\LEQ}{\leqslant}
\newcommand{\GEQ}{\geqslant}
\renewcommand{\eps}{{\varepsilon}}

\renewcommand{\div}{{\rm div\,}}
\newcommand{\sign}{{\rm sign\,}}
\renewcommand{\a}{\alpha}
\renewcommand{\b}{\beta}
\renewcommand{\d}{\delta}
\renewcommand{\c}{\gamma}

\newcommand{\f}{{\varphi}}
\newcommand{\p}{\partial}
\newcommand{\s}{\sigma}
\renewcommand{\t}{\theta}
\renewcommand{\O}{\Omega}
\newcommand{\R}{{\rm I\! R}}
\newcommand{\N}{{\rm I\! N}}
\newcommand{\LdN}{{\mathcal L}^{2,N}}
\newcommand{\D}{{\mathcal D}}
\newcommand{\V}{{\mathcal V}}
\newcommand{\calP}{{\mathcal P}}
\newcommand{\WdLpq}{{W^2L^{p,q}(\O)}}
\newcommand{\Lpq}{{L^{p,q}(\O)}}
\newcommand{\LNu}{{L^{N,1}(\O)}}
\newcommand{\WdLNu}{{W^2L^{N,1}(\O)}}
\newcommand{\LNua}{{L^{\frac N{N-1+\a}}(\O)}}
\newcommand{\ORA}{\vec}
\renewcommand{\OV}{\overline}
\newcommand{\curl}{{\rm curl\,}}
\renewcommand{\exp}{{exp}}
\newcommand{\measure}{{\rm measure\,}}
\renewcommand{\ln}{{\rm Log\,}}
\newcommand{\bmor}{{{\rm bmo}_r}}
\newcommand{\bmo}{{{\rm bmo}}}
\newcommand{\loc}{{\rm loc}}
\newcommand{\hin}{\hbox{ in }}
\newcommand{\eqd}{(\ref{eq2})\,} 
 
\newcommand{\eqq}{(\ref{eq4})\,} 
\def\Max{\mathop{\rm Max\,}}

\title {Linear diffusion with singular absorption potential and/or unbounded convective flow: the   weighted space approach }

\date{}

\begin{document}
	
\maketitle

\centerline{\scshape J.I. D\'iaz and D. G\'omez-Castro}
\medskip
{\footnotesize
 \centerline{Instituto de Matematica Interdisciplinar \& Dpto. de Matematica Aplicada,}
   \centerline{Universidad Complutense de Madrid}
   \centerline{ Plaza de las Ciencias, 3, 28040 Madrid, Spain}
   \centerline{ {jidiaz@ucm.es}, {dgcastro@ucm.es}}
} 

\medskip

\centerline{\scshape Jean Michel Rakotoson\footnote{corresponding author}}
\medskip
{\footnotesize
 \centerline{ Universit\'e de Poitiers }
   \centerline{Laboratoire de Math\'ematiques et Applications - UMR CNRS 7348 - SP2MI, France}
   \centerline{Bd Marie et Pierre Curie,  T\'el\'eport 2,  F-86962 Chasseneuil Futuroscope Cedex, France}
   \centerline{Jean-Michel.Rakotoson@math.univ-poitiers.fr}
}

\medskip

\centerline{\scshape Roger Temam}
\medskip
{\footnotesize
	\centerline{  Institute for Scientific Computation and Applied Mathematics  }
	\centerline{  Indiana University }
	\centerline{  Bloomington, Indiana 47405, U.S.A. }
	\centerline{ temam@indiana.edu}
}

\bigskip

\begin{abstract}
	In this paper we prove the existence and uniqueness of very weak solutions to linear diffusion equations involving a singular absorption potential and/or an unbounded convective flow on a bounded open set of $\R^N$. In most of the paper we consider homogeneous Dirichlet boundary conditions but we prove that when the potential function grows faster than the distance to the boundary to the power -2 then no boundary condition is required to get the uniqueness of very weak solutions. This result is new in the literature and must be distinguished from other previous results in which such uniqueness of solutions without any boundary condition was proved for degenerate diffusion operators (which is not our case). Our approach, based on the treatment on some distance to the boundary weighted spaces, uses a suitable regularity of the solution of the associated dual problem which is here established. We also consider the delicate question of the differentiability of the very weak solution and prove that some suitable additional hypothesis on the data is required since otherwise the gradient of the solution may  not be integrable on the domain. 
\end{abstract}

\section{Introduction}
In this paper we want to develop a weighted space approach to study the existence, uniqueness and regularity of linear diffusion equations involving singular and unbounded coefficients of the type
\begin{equation}\label{eq1}
-\Delta\omega+\ORA u\cdot \nabla\omega+V \omega=f \hbox{ on }\O,
\end{equation}
where $V$ is a  very singular potential being in general non negative and  locally integrable. 
To fix ideas, we shall consider mainly the case of Dirichlet boundary conditions
\begin{equation}\label{eq1*}
\omega=0\hbox{ on }\p\O,
\end{equation}
but  our weighted space approach can  also be adapted to the case of Neumann boundary conditions and, what is more remarkable, to the case of {\bf no~boundary~conditions~on}~$\p\O$ (but still getting the uniqueness of solutions) for some specially singular potentials (see the  subsection \ref{ss322} in section 4). Here $\O$ is an open bounded smooth (for instance with $\p\O$ of class $C^{2,1}$) of $\R^N$, $N \GEQ 2$, (the case $N=1$ and $u=$constant is considerably simpler)  . 
The external forcing term $f(x)$ will be assumed such  that
\begin{equation}\label{eq1II}
f\in L^1(\O;\d)\end{equation}
where the weight in this space is given by
\begin{equation}\label{eq1III}
\d(x)=d(x,\p\O)
\end{equation}
(sharper results will require some slight restrictions to (\ref{eq1II}) (see for instance  section \ref{s4}). We recall that (\ref{eq1II}) is optimal in the cases $V\equiv 0$ and $\ORA u=\ORA 0$ as it can be shown by  explicitly computing the  Green kernel for special domains.\\
Although we shall indicate later the detailed assumptions on the data, we anticipate now that we shall always assume that the convective flow vector $\ORA u $ satisfies
\begin{equation}\label{eq1IV}
\begin{cases}
\ORA u\in L^N(\O)^N,\ \div\ORA u=0 &\hin\ \ {\mathcal D'}(\O)\hbox{ and }\\
\ORA u\cdot\ORA n=0&\hbox{on\ }\p\O
\end{cases}
\end{equation}
where $\ORA n$ denotes the unit exterior normal vector to $\p\O$. Notice that, due to (\ref{eq1IV}), the weak solution notion adapted to equation (\ref{eq1}) is equivalent to the  one defined for the treatment of the equation in divergent form that is
\begin{equation}\label{eq1V}
-\Delta \omega+\div(\ORA u\,\omega)+V\omega=f\quad \hin\ \O.
\end{equation}
It is well-known that the mathematical treatment  of diffusion equations such  as (\ref{eq1}) \big(or (\ref{eq1V})\big) leads to quite satisfactory results (in view of some applications) when the data $f$, $\ORA u$ and $V$ are assumed to be bounded. Nevertheless, the main interest of this work concerns the limit cases in which $V(x)$ is assumed to be a singular function (mainly with its singularity located on $\p\O$) and/or when $\ORA u$ is an unbounded vector (satisfying (\ref{eq1IV})). Let us indicate some relevant applications leading to the consideration of such limit cases~:
\begin{enumerate}
	\item {\it The vorticity equation in fluid mechanics.} Equation (\ref{eq1}) can be derived from the stationary Navier-Stokes in 2D
	\begin{equation}\label{eq1VI}
	-\Delta\ORA u+(\ORA u\cdot \nabla)\ORA u+\nabla p=\ORA F
	\end{equation}
	taking the curl of the equation and setting
	\begin{equation}\label{eq1VII}
	f= \ORA F\cdot\ORA k,\qquad \omega=\curl\ORA u\cdot\ORA k,
	\end{equation}
	where $\ORA k$ is the last element of the canonical basis in $\R^3$ (see e.g. \cite{Temam1}). Nevertheless, as far as we know no satisfactory theory is available in the literature under the general condition that $\ORA F\cdot\ORA k\in L^1(\O;\d)$.
	\item{\it  Schr\"odinger equation with singular potentials}.
	It is well-known that the consideration of the bound states $\psi(x,t)=e^{-iEt}\omega(x)$ leads to the stationary Schr\"odinger equation
	\begin{equation}\label{eq1VIII}
	-\Delta \omega+V(x)\omega=E\omega\qquad \hin\ \R^N.
	\end{equation}
	The Heisenberg  uncertainty principle makes specially interesting the consideration of potentials which are critically singular on $\p\O$ more precisely, such that
	\begin{equation}\label{eq1IX}
	V(x)\GEQ \dfrac c{\d(x)^2},\qquad a.e\ x\in\O,
	\end{equation}
	for some $c>0$, which implies that $\omega=\dfrac{\p\omega}{\p\ORA n}=0$ on $\p\O$, so that we  can assume that $\omega\equiv 0$ on $\R^N-\O$ (see \cite{Dinter, 12b}). Here we shall not consider any eigenvalue problem like (\ref{eq1VIII}) but the study of (\ref{eq1}) for potentials $V(x)$ satisfying (\ref{eq1IX}) will be very useful for later works in this direction.
	\item{\it Linearization of singular and/or degenerate nonlinear equations}. For many different purposes, it is very convenient to ``approximate'' the solutions of quasilinear diffusion equations of the type 
	\begin{equation}\label{eq1X}
	-\Delta \f(w)+\div\big(\ORA\phi(w)\big)+g(w)=f(x)\qquad\hin\ \O
	\end{equation}
	by the solutions of the associated linearized equation. This is what appears, for instance, in the study of the stability of the associated parabolic or hyperbolic equations and also in some control problems associated with (\ref{eq1X}). Usually, it is assumed  that $\f$ is a strictly increasing function. So by considering $\t:=\f(w)$ we get
	\begin{equation}\label{eq1XI}
	-\Delta\t+\div\big(\ORA\psi(\t)\big)+h(\t)=f(x)\qquad\hin\ \O,
	\end{equation}
	with
	\begin{equation}\label{eq1XII}
	\begin{cases}
	\ORA\psi:\R\to\R^N,\quad\ORA\psi=\ORA\phi\circ\f^{-1},\\
	h=g\circ\f^{-1}.
	\end{cases}
	\end{equation}
	
	Now, assume that $\t_\infty(x)$ is a given solution of (\ref{eq1XI}), satisfying, for instance, $\t_\infty=0$ on $\p\O$. Then the  ``formal linearization'' of equation (\ref{eq1XII}) around the solution $\t_\infty(x)$ coincides with equation (\ref{eq1}) when we take
	\begin{equation*}\ORA u(x):=\ORA\psi\big(\t_\infty(x)\big)\end{equation*}
	and 
	\begin{equation*}V(x)=h'\big(\t_\infty(x)\big).\end{equation*}
	What makes difficult the study of the corresponding problem (\ref{eq1}) is the fact that in
	many cases relevant in the reaction-diffusion theory (see e.g. \cite{HernandezManceboVega2002}) functions $\ORA\psi'(r)$
	and $h'(r)$ present a singularity at $r=0$ and so, at least on $\p\O$, the coefficients $\ORA u$ and $\ORA V$ are singular. A qualitative information on the behavior of $\t_\infty(x)$ near $\p\O$ allows us  to get the precise information about the singularities of $\ORA u$ and/or $V$ near $\p\O$ \big(which, for instance, is of the type  (\ref{eq1IX})\big).
	
	\item{\it Shape optimization in Chemical Engineering}. When dealing with the problem of shape optimization for chemical reactors and applying technics of shape differentiation,  it was shown that if $g\in W^{2,\infty}(\R)$, then the solutions $u_0$ of the problem
	\begin{equation}\label{eq1XIII}
	\begin{cases}
	-\Delta u+g(u)=f,&\O,\\
	u=1,&\p\O,\end{cases}
	\end{equation}
	are differentiable with respect to the domain in the sense of Hadamard \cite{Had}  and  after developed in Murat and  Simon \cite{MS, MS2}    and the derivative $u'$ in the direction of a deformation $\t\in W^{1,\infty}(\R^n,\R^n)$ is the solution of the problem
	\begin{equation}
	\begin{cases}
	-\Delta u'+g'(u_0)u'=0,\\
	u'+\t\cdot\nabla u\in H^1_0(\O).
	\end{cases}
	\end{equation}
	Applying the theory developed for the  general case (\ref{eq1}), we can give a meaning to the shape derivative  if the domain   is not smooth as, for example, for root type kinetics (see \cite{DG, GC}). These nonlinear terms $g(u)$ are known in chemistry as Freundlich kinetics and have signifiant importance. Once again, taking $V(x)\equiv g'\big(u_0(x)\big) $ we arrive to problem~(\ref{eq1}).
	
\end{enumerate}

Some  previous papers dealing with data in $L^1(\O;\d)$ and/or singular potentials (with usually   $\ORA u=\ORA 0$) are 
\cite{DR1, DRJFA, RakoJFA, AbergelRako, MerkerRako, RabookLinear, VV} 
(see also the references therein).\\

We also mention that sometimes it is possible to get conclusions for the stationary problem (\ref{eq1}) (with $\ORA u=\ORA 0$) through the consideration of the associated
evolution equations (see e.g.  \cite{BrCa98}, \cite{BrezisKato} and its references).\\

In this paper we shall work with the notion of  ``very weak solutions'' (v.w.s.) of problem~(\ref{eq1}).
\begin{defi}\label{debdef}{\bf (Very weak solutions of problem (\ref{eq1})).}\ \\
	Let $f$ be in $L^1(\O;\d)$ and $\ORA u\in L^{N,1}(\O)^N$ with $\div(\ORA u)=0$ in $\D '(\O)$, $\ORA u\cdot\ORA n=0$ on $\p\O$,  $V$  measurable and  non negative function. A very weak solution $\omega$ of (\ref{eq1}) is a function $\omega\in L^{N',\infty}(\O)$ satisfying
	\begin{equation}\label{eq1XV}
	V\omega\in L^1(\O;\d)\hbox{ and } \int_\O\omega\big[-\Delta \phi-\ORA u\cdot\nabla\phi+V\phi\big]dx=\int_\O f\phi\, dx,
	\end{equation}
	for all $\phi\in C^2(\OV\O)$ with $\phi=0$ on $\p\O$, if $V\in L^1(\O;\delta)$, or for all $\phi \in C^2_c(\O)$ if $V\in L^1_{loc}(\O).$\\
\end{defi}

Notice that we look for a function in the  space $L^{N',\infty}(\O)$ where $N'=\frac N{N-1}$ instead of $\omega~\in ~L^1(\O)$ as usual,  in order to get more general assumptions on  $\ORA u$ and $V$.\\

We also also point out that our study will be concentrated in the case of  ``absorption'' potentials $V(x)\GEQ 0$ a.e. $x\in\O$. In  fact, as we shall see later, the study is also applicable to some general potentials  such that e.g. $V(x)\GEQ-\lambda$ with $0<\lambda<\lambda_1$ ($\lambda_1$ being the first eigenvalue of the Laplacian on $\O$ with zero Dirichlet boundary condition). As we shall show, this does not induce a restriction on the growth of the singularity of such absorption potentials near $\p\O$ (in contrast with the well-known results for {\bf negative} potentials, see e.g. \cite{BrCa98}).

The detailed definition of the Lorentz spaces $L^{p,q}(\O)$ and some other spaces which  we shall use in our study will be the object of Section 2 of this paper. Other preliminary
results and the statement of some of our main conclusions will be also presented there.\\

The proof of the existence  and uniqueness of a very weak solution (v.w.s.) for (\ref{eq1})  needs a deep study of the dual problem associated with (\ref{eq1})
\begin{equation}\label{eq1XVI}
\begin{cases}
-\Delta\phi-\ORA u\cdot\nabla \phi+V\phi=T &\hin\ \ \O,\\
\phi=0&\hbox{ on\ }\p\O.
\end{cases}
\end{equation}
Notice the change of sign  in the convection term.  We anticipate that in some cases no boundary condition will be assumed on $\phi$.\\

In Section 3,  we discuss, depending on $V$ and $\ORA u$, the existence and the regularity of the solution of the dual problem. After this, we shall be concerned with the existence of the very  weak solution in $L^{N',\infty}(\O)\cap L^1(\O;V\d)$, when $V\GEQ0$ is locally integrable. We will show that the very weak  solution $\omega$ of equation (\ref{eq1}) under zero Dirichlet  or   Neumann  boundary condition  has its  gradient in the Sobolev-Lorentz weighted space   $W^1L^{1+\frac1N,\infty}(\O;\d)$
in particular we shall get the estimate
\begin{equation}\label{eq1XVII}
\int_{\{x:|\nabla\omega|(x)<\lambda\}}\delta(x)dx\LEQ\dfrac{constant}{\lambda^{1+\frac1N}}\hbox{ for all }\lambda>0,
\end{equation}
under the mere assumption $\ORA u\in L^{N,1}(\O)^N$. Thus, we can conclude that $\nabla \omega\in L^1_{loc}(\O)$.\\

The question of uniqueness of v.w.s. given by (\ref{eq1XV}), when $V$ is only in $L^1_{loc}(\O)$ is one of the major difficulties in this general framework. When $V$ is sufficiently integrable, say $V\in L^{N,1}(\O)$, then we derive the uniqueness thanks to the regularity of the dual problem. If $V$ is only locally  integrable, but $V$ is bounded from below by $c\d^{-r}, \ r>2$ near the boundary, then the v.w.s. is unique even when no boundary condition is specified on $\p\O$ (but we additionally know that $V\omega\in L^1(\O;\d)$).\\
The uniqueness proof relies on the  $L^1(\O;\d)$-accretiveness property of the operator (see \cite{Ra2})  $T\OV \omega=-\Delta \OV\omega +\div(\ORA u\,\OV\omega)$ when $\OV\omega\in L^1(\O;\d^{-r})\cap W^{1,1}_{loc}(\O)$. This is given through the following local version of  the Kato's type inequality
\begin{equation}\label{eq1XVIII}
\int_\O\OV\omega_+T^*\psi\,dx\LEQ\int_\O\psi\,\sign_+(\OV\omega)T\OV\omega\,dx,\hbox{ whenever }T\OV\omega\in L^1_{loc}(\O),\ \psi\in\D(\O),
\end{equation}
and a special approximation of test function $\f$ in $C^2(\OV\O)$ by a sequence of functions of the type $\f_n(x)=\d(x)^rh_n(x)$ with $h\in C^2_c(\O)$ and $r>0$ (see Lemma \ref{l800}).
We point out that, besides the concrete interest of (\ref{eq1XVIII}) in itself; such  an inequality has many consequences since it allows to apply the semigroup operators theory on suitable functional spaces.\\
Concerning very weak solutions (where no differentiability is asked to the function $\omega$), a natural question (originally set by H. Br\'ezis in 1972 when $\ORA u=0$) is then : when should we have $|\nabla \omega|$ in $L^1(\O)$? The answer to this question will require some suitable additional  integrability conditions on $f$ and  $\ORA u$.\\

Note that   for proving some additional integrability for the very weak solutions $\omega$ is a delicate task. Indeed, we shall show that for some special cases of $\ORA u\in C^{0,\a}(\OV\O),\ \a>0$, there exists $f\in L^1_+(\O;\d)$ such that $||\omega||_{L^{N'}}=+\infty$ when $N\GEQ3$.This leads to an additional question :  under what conditions could we improve the integrability of $\omega$, to say $\omega\in L^{N'}(\O)$? The answer to this  question is also one of the main results of this paper.\\

Before stating the study of the main equation (\ref{eq1}), we shall recall some notations and  functional spaces that we shall use.
\section{Notations, preliminary definitions and results}
Before stating our main results concerning equation (\ref{eq1}) we  need to recall some notations and some functional spaces which are relevant for the study of the ``dual problem'' (\ref{eq1XVI}) 
under very general regularity assumptions on the coefficients $\ORA u$ and $T$.

\begin{defi}{\bf(  $\bmo(\R^N)$)} \cite{gold}.\\
	A locally integrable function $f$ on  $\R^N$ is said to be in $\bmo(\R^N)$ if
	\begin{equation*}
	\!\!\!\!\!\!\!\!\!\!\!\!\sup_{0<{\rm diam\,}(Q)<1}\frac1{|Q|}\int_Q|f(x)-f_Q|\,dx+\!\!\!\!\!\!\sup_{ {\rm diam\,}(Q)\GEQ1}\frac1{|Q|}\int_Q|f(x)|dx
	\end{equation*}
	\begin{equation*}
	\equiv||f||_{\bmo(\R^N)}<+\infty,
	\end{equation*}
	where the supremum is taken over all cube $Q\subset\R^N$ the sides of which are parallel to the coordinates axes.\\
	Here $\DST f_Q=\dfrac1{|Q|}\int_Qf(y)dy$.
\end{defi}

\begin{defi}{\bf( $\bmo_r(\O)$ )} \cite{Chang, CDS}.\\
	A locally integrable function $f$ on a Lipschitz bounded domain $\O$ is said to be in $\bmo_r(\O)$  ($r$ stands for restriction) if 
	\begin{equation}\label{eqPW1}
	\!\!\!\!\!\!\!\!\!\!\!\!\sup_{0<{\rm diam\,}(Q)<1}\frac1{|Q|}\int_Q|f(x)-f_Q|\,dx
	+\int_\O|f(x)|dx
	\equiv||f||_{\bmo_r(\O)}<+\infty,\end{equation}
	where the supremum is taken over all cube $Q\subset\O$ the sides of which are parallel to the coordinates axes.\\
	In this case, there exists a function $\WT f\in \bmo(\R^N)$ such that
	\begin{equation}\label{eqPW2}
	\WT f\Big|_\O=f\hbox{ and } ||\WT f||_{\bmo(\R^N)}\LEQ c_\O\cdot||f||_{\bmo_r(\O)}.
	\end{equation}
\end{defi}
\begin{rem}\ \\
	The above definition  adapted to the case where the domain $\O$ is  bounded,   is equivalent to the definition given in \cite {CDS, Chang}. The main property (\ref{eqPW2}) is due to  P.W Jones \cite{J}.\\
	This extension result implies that $\bmo_r(\O)$ embeds continuously into $L_{exp}(\O)$ (a space  which we shall introduce below in Definition \ref{d5}.)
\end{rem}

\begin{defi}{\bf (Campanato space $\LdN(\O)$.)}\\  
	A function $u\in\LdN(\O)$ if
	\begin{equation*} ||u||_{L^2(\O)}+\sup_{x_0\in \O, r>0}\Big[r^{-N}\int_{Q(x_0,r)\cap\O}|u-u_r|^2\,dx\Big]^{\frac12} :=||u||_{\LdN(\O)}<+\infty.
	\end{equation*}
	Here
	\begin{equation*}
	u_r:=\frac1{|Q(x_0;r)\cap\O|}\int_{Q(x_0;r)\cap\O}u(x)\,dx.
	\end{equation*}
\end{defi}
In fact the two above definitions are equivalent :
{}{
	\begin{theo}\label{t9}{\cite{RabookLinear}}\\
		For a Lipschitz bounded domain $\O$ one has 
		\begin{equation*}\LdN(\O)=\bmo_r(\O),\hbox{ with equivalent norms.}\end{equation*}
	\end{theo}
}
We set \begin{equation*}L^0(\O)=\Big\{v:\O\to\R\hbox{ Lebesgue measurable}\Big\}\end{equation*} and we denote    by $L^p(\O)$ the usual Lebesgue space $1\LEQ p\LEQ +\infty$.
Although it is not too standard, we shall use the notation $W^{1,p}(\O)=W^1L^p(\O)$ for the  associate Sobolev space. We shall need the following definitions:
\begin{defi}{\bf (of the distribution function and monotone rearrangement.)}\\
	Let $u\in L^0(\O)$. The distribution function of $u$ is the decreasing function 
	\begin{equation*}m=m_u:\R\mapsto[0,|\O|] \end{equation*} 
	\begin{equation*}m_u=m_u(t)=\measure\big\{ x:u(x)>t\big\}=|\big\{u>t\big\}|.\end{equation*}
	The generalized inverse $u_*$ of $m$ is defined by
	\begin{equation*}u_*(s)=\inf\Big\{t:|\big\{u>t\big\}|
	\LEQ s\Big\},\quad s\in[0,|\O|[\end{equation*}
	and is called the decreasing rearrangement of $u$. We shall set $\O_*=]0,|\O|\,[.$
\end{defi}
We recall now the following definitions :
{}{
	\begin{defi}\label{d5}\ \\
		Let $1\LEQ p\LEQ +\infty,\ 0<q\LEQ+\infty$ :\\
		$\bullet$ If $q<+\infty$, one defines the following norm for $u\in L^0(\O)$
		\begin{equation*}||u||_{p,q}=||u||_{L^{p,q}}:=\left[\int_{\O_*}\left[t^{\frac1p}|u|_{**}(t)\right]^q\frac{dt}t\right]^{\frac1q}\hbox{ where }|u|_{**}(t)=\dfrac1t\int_0^t|u|_*(\s)d\s.\end{equation*}
		$\bullet$ If $q=+\infty$,
		\begin{equation*}||u||_{p,\infty}=\sup_{0<t\LEQ|\O|}t^{\frac1p}|u|_{**}(t).\end{equation*}
		The space $L^{p,q}(\O)=\Big\{ u\in L^0(\O):||u||_{p,q}<+\infty\Big\}$ is called a {\bf Lorentz space}.\\
		$\bullet$ If $p=q=+\infty,$
		\qquad$L^{\infty,\infty}(\O)=L^\infty(\O).$\\
		The dual of $L^{1,1}(\O)$ is called $L_\exp(\O)$
		
	\end{defi}
}
\begin{rem}
	We recall that $L^{p,q}(\O)\subset L^{p,p}(\O)=L^p(\O)$ for any $p>1,\ q\GEQ1$.
\end{rem}
For $\a>0$, we define
\begin{align*}
	L^\a_\exp(\O)&=\left\{v:\O\to\R,\ \DST\sup_{0<s<|\O|}\dfrac{|v|_*(s)}{\left(1-\ln\dfrac s{|\O|}\right)^\a}<+\infty\right\},\\
	L^p(\ln L)^\a&=\left\{f:\O\to\R,\ \DST\int_{\O_*}\left[\left(1-\ln\dfrac s{|\O|}\right)^\a|f|_*(s)\right]^pds<+\infty\right\}.
\end{align*}
When there is no  possible confusion, we denote by the same symbol the space product $V^N$ and $V$.\\
We recall also that if $v,\ u\in L^1(\O)$, then 
\begin{equation*}\DST v_{*u}\dot=\lim_{\lambda\searrow 0}\frac{(u+\lambda v)_*-u_*}\lambda\end{equation*} exists in a weak sense and it is called the relative rearrangement of $ v$ with respect to $u.$ 
More precisely, we have the following result  (see   \cite {MT, RakoRR}).
\begin{theo}\label{t1t1}
	\ \\ Let $\O$ be a bounded measurable set in $\R^N$\!, $u$ and $v$ two functions in $L^1(\O)$  and let  $w~:~\overline\O_*~\to~\R$ be defined by:
	\begin{equation*}w (s)=\int_{\{u>u_*(s)\}}\hspace{-1cm}v(x)dx
	+\int_0^{s-|u>u_*(s)|} \Big(v\Big|_{\{u=u_*(s)\}}\Big)_*(\sigma)d\sigma,\end{equation*}
	\hbox{where }$v\Big|_{\{u=u_*(s)\}}$\hbox{ is the restriction of $v $ to $\{u=u_*(s)\}$.}\\
	Then 
	\begin{equation*}\DST
	\frac{(u+\lambda v)_*-u_*}\lambda\underset{\lambda\to0}\rightharpoonup
	\frac{d w }{ds}\hin\  \begin{cases}L^p(\O_*)\hbox{-weak}&if\  v\in L^p(\O),\ 1\LEQ p<+\infty\\ 
	L^\infty(\O_*)\hbox{-weak-star}&if\ v\in L^\infty(\O)\end{cases}.\end{equation*}
	Moreover,  $\DST\left|\frac{dw }{ds}\right|_{L^p(\O_*)}\LEQ |v|_{L^p(\O)}.$
\end{theo}

One  property that we shall use for the relative rearrangement is the following one:
\begin{propo}\label{pp1}\ \\
	Let $v\GEQ0$, and $u$ be two functions in $L^1(\O)$. Then \begin{equation*}(v_{*u})_{**}\LEQ v_{**}.\end{equation*}
\end{propo}
There is a link between the derivative of $u_*$ and the relative rearrangement of the gradient of $u$ as it was proved in \cite{RakoRR, RT}. We will use only the following result (see \cite{RakoRR})

\begin{theo}\label{t10}\ \\
\vspace{-0.5cm}
	\begin{itemize}
		\item[(a)] Let $ u\in W^{1,1}_0(\O), u\GEQ0$. Then
		\begin{equation*}-u'_*(s)\LEQ\frac{s^{\frac1N-1}}{N\alpha_N^{\frac1N}}|\nabla u|_{*u}(s)\hbox{\quad a.e in } \O_*,\end{equation*}
		and \begin{equation*}-u'_{**}(s)\LEQ \frac{s^{\frac1N-1}}{N\alpha_N^{\frac1N}}(| \nabla u|_{*u})_{**}(s)\ a.e. \hin\ \O_*.\end{equation*}
		\item[(b)] Let $u\in W^{1,1}(\O)$. Then if $\O$ is a Lipschitz connected open set of $\R^n$
		\begin{equation*}-u'_*(s)\LEQ\frac{\min(s,|\O|-s)^{\frac1N-1}}{Q(\O)}|\nabla u|_{*u}(s),\end{equation*}
		where $Q(\O)$ is a  suitable constant depending only on $\O.$
	\end{itemize}
\end{theo}
Note that $u_*$ is in $W^{1,1}_{loc}(\O_*)$ under statements (a) and (b) (see \cite{RakoRR, RT}).\\

Let $V$ be a Banach space contained in $L^1_{loc}(\O)$.The norm on $V$ is denoted by $||\cdot||_V$ (or simply $||\cdot||$). We define the Sobolev space over $V$, for $m\in\N$ by 
\begin{equation*}W^mV=\Big\{ v\in L^1_{loc}(\O):D^\a v\in V\hbox{ for any } |\a|=\a_1+\ldots+\a_N\LEQ m\Big\}.\end{equation*}
In particular, $W^1_0V=W^1 V \cap W^{1,1}_0(\O)$.\\

The following density result can be found in   \cite{g1, RakoRome, RabookLinear}:
\begin{theo}\label{tt4}{\bf (Density)}\\
	Let $\O$ be a bounded set of class $C^{1,1}$. Then, the set $\{\f\in C^2(\OV\O):\f=0\hbox{ on }\p\O\Big\}$ is dense in 
	$\Big\{\f\in W^2L^{p,q}(\O):\f=0\hbox{ on }\p\O\Big\},\ 1<p<+\infty,\ 1\LEQ q\LEQ +\infty.$
\end{theo}
 
 \begin{rem}\label{remImp}\ \\ 
	Here and along the paper $\ORA u$ is at least in $L^N(\O)^N,\ \div(\ORA  u)=0\hbox{ in }\D'(\O)$ and $\ORA u\cdot\ORA n=0$ on $\p\O$, if $N\GEQ3$ and $\ORA u\in L^{2+\eps}(\O),$  for some  $\eps>0$ if $N=2$.
	The value of $\ORA u\cdot\ORA n$ on $\p\O$ is defined through the Green's formula (see \cite{Temam1}).
\end{rem}

The following density result can be proved using the same argument as for the  $L^p$-case (see \cite{Temam1, CF})
\begin{propo}\label{p1}{\bf (Density of smooth functions).}\\
	Let $1<p<+\infty$ and  $1\LEQ q\LEQ\infty$. Then the closure of the set 
	\begin{equation*}{\mathcal V}=\Big\{\ORA u\in C^\infty_c(\O)^N:\div(\ORA u)=0\hbox{ in }\O\Big\}\end{equation*}
	in $L^{p,q}(\O)^N$ (resp. $(L^N(\ln L)^\a)^N,\ \a>0$ ) is the space
	\begin{equation*}\OV\V\!:=\!\Big\{\ORA u\in L^{p,q}(\O)^N\hbox{\,(resp.\,$(L^N(\ln\,L)^\a)^N,\,\a>0$\,)}\,:\div(\ORA u)=0\hbox{\,and\,} \ORA u\cdot\ORA n=0\hbox{\,on\,}\p\O\Big\}.\end{equation*}
\end{propo}
Due to Proposition \ref{p1}, a standard approximation argument leads to :
\begin{lem}\label{l1}\ \\
	For all Lipschitz mappings $G:\R\to\R$,  and for all $\phi\in W^1_0L^{N'}(\O)$ with $N'=\dfrac N{ N-1}$, one has
	\begin{equation*}\int_\O(\ORA u\cdot\nabla \phi)\,G(\phi)\,dx=0.\end{equation*}
\end{lem}
\begin{lem}\label{l2}\ \\
	For all $\OV\omega\in H^1_0(\O)$, and for all $\phi\in H^1_0(\O)$
	\begin{equation*}\int_\O(\ORA u\cdot\nabla\OV\omega)\,\phi\,dx=-\int_\O\ORA u\cdot\nabla\phi\,\OV\omega\,dx.\end{equation*}
\end{lem}
Let us remark that,\\
$\bullet$ if $N\GEQ3$ 
\begin{equation}\label{eqrak}\left|\int_\O\ORA u\cdot\nabla \OV \omega\,\phi \,dx\right|\LEQ ||\ORA u||_{L^N}||\nabla\OV\omega||_{L^2}||\phi||_{L^{2^*}}\hbox{ 
	where }\dfrac1{2^*}+\dfrac12+\dfrac1N=1,\end{equation} 
$\bullet$ if $N=2$  the above  inequality holds true after replacing $N$ by $2+\eps$ and $2^*$ by $\dfrac{2(2+\eps)}\eps$.\\

We shall need the following classical result (see  \cite{Lions}) :
\begin{lem}\label{l2l}\ \\
	Let $\DST X\ {\hookrightarrow}_c\ Y\hookrightarrow Z$ be three Banach spaces each continuously embedded in the next one, the first inclusion is supposed to be compact. Then, for all $\eps>0$ there exists a constant $c_\eps>0$ such that $\forall\,\phi\in X$
	\begin{equation*}||\phi||_Y\LEQ\eps||\phi||_X+c_\eps||\phi||_Z.\end{equation*}
\end{lem}

\section{Existence,\,uniqueness,\,regularity\, and results\,for\,the\,dual\,problem}
\subsection{Case where the potential $V$ is only measurable and bounded from below}
We  first   study the solvability of  the dual problem (\ref{eq1XVI}) (equivalent to (\ref{eq2}) below  and the regularity of its solutions.\\
The following result,  consequence of the Lax-Milgram theorem, is a remarkable fact due to the low regularity assumed on the data $\ORA u$ and $V$:
\begin{propo}\label{p2}\ \\
	Let $T\in H^{-1}\!(\O)$\,(dual space\,of\,$H^1_0\!(\O)$), $\ORA u$ satisfying (5) and let $V\in L^0(\O)$  satisfying $V\GEQ-\lambda$ for some $\lambda \in[\,0,\lambda_1)$ where 
	$\lambda_1$ is the first eigenvalue of $-\Delta$ under the zero Dirichlet  boundary condition. Define $W=\Big\{\f\in H^1_0(\O):(V+\lambda)\f^2\in L^1(\O)\Big\},$ and let  $W'$ denotes its dual.\\Then, there exists a unique  $\phi\in H^1_0(\O)$, with $(V+\lambda)\phi^2\in L^1(\O)$, such that
	\begin{equation}\label{eq2}
	(\calP)_{V,T}\qquad-\Delta\phi-\ORA u\cdot\nabla\,\phi+V\phi=T\hbox{ in }W'.
	\end{equation}
	
	Moreover, \begin{equation*}||\phi||_{H^{1}_0(\O)}=\left(\int_\O|\nabla\phi|^2dx\right)^{\frac12}\LEQ\dfrac{\lambda_1}{\lambda_1-\lambda}||T||_{H^{-1}(\O)},\end{equation*}
	\begin{equation*}\left(\int_\O(V+\lambda)\phi^2dx\right)^{\frac12}\LEQ\left(\dfrac{\lambda_1}{\lambda_1-\lambda}\right)^{\frac12}||T||_{H^{-1}(\O)},\end{equation*}
	\begin{equation*} V\phi\in L^1_{loc}(\O).\end{equation*}
	If furthermore $V\in L_{loc}^1(\O)$, then the equation (\ref{eq2}) holds in the sense of distributions in $\D'(\O)$
	
\end{propo}
\begin{proof}
 
 We endow     $W$ with the following norm  \begin{equation*}[\f]^2_W=||\f||^2_{H^1_0(\O)}+\int_\O(V+\lambda)\f^2dx.\end{equation*}
Let us consider the bilinear form on $W$ given by
\begin{eqnarray*}
	a(\psi,\f)&=&\int_\O\nabla\psi\cdot\nabla\f dx-\int_\O\ORA u\cdot\nabla\psi \f dx
	+\int_\O(V+\lambda)\psi\f dx\\
	&&-\lambda\int_\O\psi\f dx,\quad (\psi,\f)\in W^2.
\end{eqnarray*}
Then, by Lemmas \ref{l1} and \ref{l2}
\begin{equation}\label{eqqq}
a(\psi,\psi)=\int_\O|\nabla\psi|^2-\lambda\int_\O\psi^2dx+\int_\O(V+\lambda)\psi^2 dx\GEQ
\a_0\left[\int_\O(V+\lambda)\psi^2+\int_\O|\nabla\psi|^2\right],
\end{equation}
with $\a_0>0.$\\

According to the above remark (\ref{eqrak}), since $\ORA u\in L^N(\O)^N$, the bilinear form is continuous on $W$ and we have
\begin{equation*}|a(\psi,\f)|\LEQ M[\psi]_W[\f]_{W}\ ,\end{equation*}
with $M=3(1+||\ORA u||_{L^N})$. Moreover, since 
$W\hookrightarrow H^1_0(\O)\hookrightarrow L^2(\O)\hookrightarrow H^{-1}(\O)\hookrightarrow W'$ we have
\begin{equation*} \langle T,\psi\rangle_{H^{-1}H^{1}_0}\LEQ ||T||_{H^{-1}}[\psi]_W, \ \forall\,\psi\in W.\end{equation*}

Thus we may apply the Lax-Milgram theorem to derive the existence of a unique $\phi\in W$, such $a(\phi,\psi)=\langle T,\psi\rangle _{H^{-1}H^{1}_0}\ \forall\,\psi\in W.$ The estimate on $\phi$ follows from (\ref{eqqq}).\\
If $V\in L^1_{loc}(\O)$ then one has \begin{equation*}\D(\O)\subset W.\end{equation*}
Moreover, since $\DST\int_\O(V+\lambda)\phi^2dx$ is finite, the Cauchy-Schwarz inequality yields

\begin{equation}\label{eq1eq1}
0\LEQ\int_{\O'}(V+\lambda)|\phi|dx\LEQ\left(\int_\O(V+\lambda)\phi^2dx\right)^{\frac12}\left(\int_{\O'}(V+\lambda)dx\right)^{\frac12}<+\infty\end{equation}
for any open set $\O'$ relatively compact in $\O$.\\
Writing 
\begin{equation*}\int_{\O'}|V\phi|dx\LEQ\int_{\O'}(V+\lambda)|\phi|dx+\lambda\int_\O|\phi|dx,\end{equation*}
the right hand is finite   taking into account (\ref{eq1eq1}) and the fact that $\phi\in L^2(\O)$ . Thus,  we have $\forall\,\O'\subset\subset\O,\ V\phi\in L^1(\O').$ 
We conclude that $V\phi\in L^1_{loc}(\O).$ 
\end{proof}

As usual in some problems of Quantum Mechanics (see e.g. Lemma 2.1 of \cite{Dinter})  it is very useful to approximate the solution $\phi\in H^1_0(\O)$ of the dual problem (\ref{eq2}) found in Proposition \ref{p2} by a sequence of solutions $\phi_k$ corresponding to a sequence of bounded potentials $V_k$ approximating $V$. Let us define $V_k$ by
\begin{equation*}V_k=\min(V,k).\end{equation*}

\begin{propo}\label{p3b}{(\bf Approximation by bounded potentials).}\\
	Let $T\in H^{-1}(\O),\ \ORA u$ and $ V$  as in Proposition \ref{p2}. Then, the sequence $\phi_k\in H^1_0(\O)$ of solutions of the problems
	\begin{equation*}(\calP)_{V_k,T}:\int_\O\nabla\phi_k\cdot\nabla\psi dx-\int_\O\ORA u\nabla \phi_k\phi dx
	+\int_\O V_k\phi_k\psi dx=\langle T,\psi\rangle,\quad\forall\,\psi\in H^1_0(\O),\end{equation*}
	converges  to $\phi$ strongly in $H^1_0(\O)$, where $\phi$ is the unique solution of $(\calP)_{V,T}$ found in Proposition \ref{p2}.
\end{propo}
\begin{proof}[Sketch of the proof of Proposition \ref{p3b}]
One has, following the   arguments of the Proposition \ref{p2}, that
\begin{equation}\label{301b}
||\phi_k||_{H^1_0}+\left(\int_\O(V_k+\lambda)\phi^2_k dx\right)^{\frac12}\LEQ2\left(\dfrac{\lambda_1}{\lambda_1-\lambda}\right)||T||_{H^{-1}(\O)}.
\end{equation}
Thus, $\phi_k$ remains in a bounded set of $H^1_0(\O)$. So we may assume that it converges to a function $\f$ weakly in $H^1_0(\O)$ and a.e. in $\O$. The above relation (\ref{301b}) implies that:
\begin{equation}\label{302b}
\left(\int_\O(V+\lambda)\f^2dx\right)^{\frac12}+||\f||_{H^1_0}\LEQ 
2\left(\dfrac{\lambda_1}{\lambda_1-\lambda}\right)||T||_{H^{-1}(\O)}.
\end{equation}
This shows that $\f\in W$  (where $W$ is the space defined in the proof of Proposition \ref{p2}). Moreover, since  for all $\psi\in W$ we have $\ORA u\psi\in L^{{2^*}'}(\O)$ (see the above remark), we deduce
\begin{equation}\label{303b}
\lim_{k\to+\infty}\int_\O\ORA u\cdot\nabla\phi_k\psi dx=\int_\O\ORA u\cdot\nabla \f\psi dx.
\end{equation}
The sequence $(V_k+\lambda)\phi_k\psi$ satisfies  Vitali's condition, since for any measurable subset $B\subset\O$, we have
\begin{equation}\label{304b}
\left|\int_B(V_k+\lambda)\phi_k\psi dx\right|\LEQ 
2\left(\dfrac{\lambda_1}{\lambda_1-\lambda}\right)||T||_{H^{-1}(\O)}
\left(\int_B(V+\lambda)\psi^2dx\right)^{\frac12}
\end{equation}
and
\begin{equation}\label{305b}
\lim_{k\to+\infty}(V_k+\lambda)(x)\phi_k(x)\psi(x)=(V+\lambda)(x)\f(x)\psi(x).
\end{equation}
Thus
\begin{equation}\label{306b}
\lim_{k\to+\infty}\int_\O(V_k+\lambda)\phi_k\psi dx=\int_\O(V+\lambda)\f\,\psi dx.
\end{equation}
We then deduce that $\f$ is solution of the problem $(\calP)_{V,T}$ and by uniqueness $\f=\phi$. Therefore,  the whole sequence $\phi_k$ converges to $\phi$  weakly   in $W$ and strongly in $L^2(\O)$.

To prove the strong convergence in $H^1_0(\O)$, let us note, using the equations $(\calP)_{V_k,T}$ and $(\calP)_{V,T}$, that
\begin{equation*}\lim_{k\to+\infty}\int_\O|\nabla \phi_k|^2dx+\int_\O(V_k+\lambda)\phi^2_kdx=\lambda\int_\O\phi^2dx+\langle\,T,\phi\,\rangle=\int_\O(V+\lambda)\phi^2+\int_\O|\nabla\phi|^2dx.\end{equation*}
Therefore, if we introduce $U_k=(\nabla\phi_k;\phi_k\sqrt{V_k+\lambda})\in L^2(\O)^{N+1}$, 
$U_\infty=(\nabla\phi;\phi\sqrt{V+\lambda})$ we have
\begin{itemize}
	\item[$\bullet$] $\DST\lim_{k\to+\infty}|U_k|^2_{L^2(\O)^{N+1}}=|U_\infty|^2_{L^2(\O)^{N+1}}$,
	\item[$\bullet$ ] $U_k$ converges  to $U_\infty$ weakly in $L^2(\O)^{N+1}$.\end{itemize}
Thus $U_k$ converges  to $U_\infty$ strongly in $L^2(\O)^{N+1}$.
\end{proof}
\begin{rem}\ \\
	{ Let us notice that for $\phi\in L^2(\O)$  the conditions $(V+\lambda	)\phi^2\in L^1(\O)$   and $|V|\phi^2\in L^1(\O)$, $\phi\in L^2(\O)$ are  equivalent. Indeed,
		since $V+\lambda=|V+\lambda|$,  
		\begin{equation*}\int_\O|V|\phi^2dx\LEQ\int_\O(V+\lambda)\phi^2dx+\lambda\int_\O\phi^2\LEQ\int_\O|V|\phi^2dx+2\lambda\int_\O\phi^2dx.\end{equation*}}\end{rem}

For this reason, from now, {\bf we will assume that } $\lambda=0$.

\begin{propo}\label{p3bb}\ \\
	Under the same assumptions as for Proposition \ref{p2} (with $\lambda=0$),
	if $T\GEQ0,\ T\in L^1(\O)\cap H^{-1}(\O)$ then $\phi\GEQ0$.
\end{propo}
 \begin{proof}
We have $\phi_-\in W$ and 
\begin{equation*}0\GEQ-\int_\O|\nabla\phi_-|dx-\int_\O V\phi_-dx=\int_\O T\phi_-dx\GEQ0.\end{equation*}
Thus 
\begin{equation*}
	\phi_-=0.
\end{equation*}
\end{proof}
For the treatment of  (\ref{eq1}) we shall need some additional regularity for the solutions of the dual problem (\ref{eq2})
independent of $\ORA u$ or $V$. We start by proving the boundedness of $\phi$ by means of some rearrangement technics   (\cite{RakoRR} p.126 of Th 5.5.1, see also  \cite{Talenti}).

We point out  that $L^{\frac N2,1}(\O)\hookrightarrow H^{-1}(\O)$.\\
\begin{propo}\label{p3}{\bf ($L^\infty$-estimates).}\ \\
	Let $\phi$ be the solution of \eqd when $T\in L^{\frac N2,1}(\O)$, $V\GEQ0$. 
	Then $\phi\in L^\infty(\O)$ and there exists a constant $K_N(\O) $ independent of $\ORA u$ and $V$~such~that
	\begin{equation*}||\phi||_{L^\infty(\O)}\LEQ K_N(\O)||T||_{L^{\frac N2,1}(\O)}.\end{equation*}
\end{propo}

\begin{proof} 
We shall argue in a way similar  to the proof of Theorem 5.3.1 in \cite{RakoRR}. 
According to Proposition \ref{p3b} , it is enough to prove the proposition for $V\in L^\infty_+(\O)$, and  and  for $T\GEQ0$, since the equation (\ref{eq2}) is linear. Thus $\phi\GEQ0$, therefore, in this proof $v=|\phi|=\phi$, but we shall keep the notation $v$ because in the general case we cannot use anymore this maximum principle.
Let $v=|\phi|,\ G_s(\s)=(\s-v_*(s))_+\ \sign(\s),$ $\s~\in~\R,\ s\in\O_*$.  The mapping $\s\mapsto G_s(\s)$ is Lipschitz. Then following Lemma \ref{l1}  
\begin{equation*}\int_\O(\ORA u\cdot\nabla\phi) G_s(\phi)\,dx=0.\end{equation*}
Therefore, we derive
\begin{equation*}\int_\O\nabla\phi\cdot\nabla G_s(\phi)=\int_{v>v_*(s)}|\nabla \phi|^2dx=\int_\O T(x)G_s(\phi)(x)dx-\int_\O V(x)G_s(\phi)dx.\end{equation*}
Differentiating this relation with respect to $s$, we find
\begin{equation*}\dfrac d{ds}\int_{v>v_*(s)}|\nabla\phi|^2dx=-v'_*(s)\int_{v>v_*(s)}\big(T(x)-V(x)\big)dx\LEQ-v'_*(s)\int_0^sT_*(\s)d\s\end{equation*}
where $T_*$ is the monotone rearrangement of $T$ (we use the fact that $V\GEQ0$).\\
Therefore, we arrive at
\begin{equation}\label{eq308}
\big[\,|\nabla\phi|^2\big]_{*v}(s)\LEQ -v'_*(s)\int_0^sT_*(\s)d\s.
\end{equation}
Since \begin{equation*}|\nabla\phi|=|\nabla v|,\hbox{\ and \ }-v'_*(s)\LEQ\dfrac{s^{\frac1N-1}}{N\a_N^{\frac1N}}|\nabla v|_{*s}(s)\end{equation*} \big(the PSR property (see Theorem 3 of \cite{RakoRR})\big) and $|\nabla v|_{*v}\LEQ\big[\,|\nabla v|^2\big]^{\frac12}_{*v}$, we infer from (\ref{eq308})
\begin{equation}\label{eq309}
-v'_*(s)\LEQ\dfrac{s^{\frac2N-2}}{(N\a_N^{\frac1N})^2}\int_0^sT_*(\s)d\s.
\end{equation}
Thus, integrating (\ref{eq309}) between 0 to $|\O|$, we find
\begin{equation*}||\phi||_{L^\infty}\LEQ c_N\int_0^{|\O|}s^{\frac2N}T_{**}(s)\dfrac{ds}s\equiv c_N||T||_{L^{\frac N2,1}(\O)}.\end{equation*}
 \end{proof}

An analogous result can be obtained when $T=-\div(\ORA F)$, with $ \ORA F\in L^{N,1}(\O)^N$.

\begin{propo}\label{p4}\ \\
	Let $N\GEQ 2$, and  let $\phi$ be a solution of \eqd when 
	$T=
	-\div(\ORA F),\ \ORA F\in L^{N,1}(\O)^N\hbox{ if\ }N\GEQ~3,$ $  
	\ORA F\in L^{2+\eps}(\O)^2\hbox{ if $N=2$}$.\\
	Then $\phi\in L^\infty(\O)$ and there exists a constant $K_N(\O)>0$ independent of $\ORA u$ and $V$ such that 
	\begin{equation*}||\phi||_{L^\infty(\O)}\LEQ K_N(\O)||\ORA F||_{L_V}\hbox{ with } L_V=
	L^{N,1}(\O)^N\hbox{ if\ }N\GEQ3,\ 
	L^{2+\eps}(\O)^2\hbox{ if }N=2.\end{equation*}
\end{propo}
\begin{proof} 
For convenience, we write $F$ for $\ORA F$. 
Thanks to Proposition \ref{p3b}, 
we can use the same test function $G_s(\phi)$ as in the proof of Proposition \ref{p3}. Then
\begin{equation*}\int_\O\nabla\phi\cdot\nabla G_s(\phi)dx+\int_\O V(x)G_s(\phi)dx=\int_\O F\cdot\nabla G_s(\phi)dx.\end{equation*}
We differentiate   this equation with respect to $s$ as before, for a.e. $s\in\O_*$, and find
\begin{equation}\label{eq310}
\big[\,|\nabla v|^2\big]_{*v}(s)-v'_*(s)\int_{v>v_*(s)}V(x)dx=\big[\,F\cdot\nabla\phi\,\big]_{*v}(s).
\end{equation}
Since, $V\GEQ0$ and $v'_*(s)\LEQ0$, we obtain
\begin{equation}\label{eq311}
\big[\,|\nabla v|^2\big]_{*v}(s)\LEQ\big[\,|\,F\,|^2\big]^{\frac12}_{*v}\big[\,|\nabla v|^2_{*v}\big]^{\frac12}(s),
\end{equation}

\begin{equation}\label{eq312}
\big[\,|\nabla v|^2\big]^{\frac12}_{*v}(s)\LEQ\big[\,|\,F\,|^2\big]^{\frac12}_{*v}(s).
\end{equation}
We have as before:
\begin{equation}\label{eq313}
-v'_*(s)\LEQ
\dfrac{s^{\frac1N-1}}
{N\a_N^{\frac1N}}\big[\,|\nabla v\,|^2\big]^{\frac12}_{*v}(s).
\end{equation}
We infer that for a.e. $s$
\begin{equation}\label{eq314}
-v'_*(s)\LEQ\dfrac{s^{\frac1N-1}}{N\a_N^{\frac1N}}\big[\,|\,F\,|^2\big]^{\frac12}_{*v}.
\end{equation}
Integrating this relation between 0 and $|\O|$ and using the  Hardy-Littlewood inequality (see  \cite{RakoRR} p.118-121) we obtain
\begin{equation*}||\phi||_{L^\infty}\LEQ\begin{cases}
c_N\DST\int_{\O_*}\s^{\frac1N-1}\big(\,|\,F\,|^2)^{\frac12}_{**}(\s)d\s,&\hbox{if }N\GEQ3,\\
c_{2,\eps}\| F\,\|_{L^{2+\eps}(\O)^2},&\hbox{if }N=2.
\end{cases}\end{equation*}
We conclude as in \cite{RakoRR} p. 118-120, Proposition 5.2.2.
\end{proof}

\begin{rem}
	The problem considered in this Section 3.1 was previously considered by
	other authors in the special case of $\overrightarrow{u}\equiv 
	\overrightarrow{0}$ (see, e.g. \cite{DalMasso Mosco 1986} and its
	references), nevertheless we emphasize that the results of this section must
	be understood as preliminary results with respect the study we shall present
	in the following sections of this paper. In particular, what is specially
	important for us is to obtain a continuous dependence estimate with respect
	to the data (namely the velocity $\overrightarrow{u}$, the potential $V,$
	and the right hand side $f$) since we need to carry out several
	perturbations of those data in the next sections. As far as we know, such
	estimates are new in the literature (and, of course, they were not given in the
	above mentioned reference).
\end{rem}

\subsection{Some regularity results with an  integrable potential $V$  and bounded from below}
As a first consequence  of Proposition  \ref{p2}   and Proposition \ref{p4} we can deduce Meyer's type regularity giving a better information on the gradient of the solution of (\ref{eq2}).
\begin{propo}\label{p5}{\bf ($W^1L^{p,q}$-estimate)}\\
	Let $N\GEQ 2$. Assume that there exists $p>N$ and $q\in[\,1,+\infty],$ such that
	\begin{equation*}\begin{cases}\hbox{$\ORA u\in L^{p,q}(\O)^N$}& V\GEQ0,\ V\in L^{r,q}(\O), \hbox{ $r=\dfrac{Np}{N+p}$},
	\\
	\hbox{$T=-\div(\ORA F)$}& \hbox{with $\ORA F\in L^{p,q}(\O)^N$}.\\
	\end{cases}\end{equation*}
	Then, the unique solution $\phi$ of the equation \eqd belongs to $W^1L^{p,q}(\O)$. 
	Moreover, there exists a constant $K_{pq}>0$ independent of $\ORA u$ such that :
	\begin{equation*}||\nabla\phi||_{L^{p,q}(\O)}\LEQ K_{pq}\left(1+||\ORA u||_{L^{p,q}}+||V||_{L^{r,q}}\right)||F||_{L^{p,q}(\O)^N}.\end{equation*}
\end{propo}
\begin{proof} 
(We shall simply  write $F,\ F_0,\ F_1$ for $\ORA F,\ \ORA F_0,\ \ORA F_1$). 
We first assume that $\ORA u\in\V$. We know from Proposition \ref{p4}  that $\phi\in L^\infty(\O)$ and that there exists a constant independent of $\ORA u,\ V$ and $\ORA F$ and $V$ such that
\begin{equation}\label{eq400}
||\phi||_\infty\LEQ K_N(\O)|| F||_{L^{p,q}(\O)}.
\end{equation}
Therefore, there exists a vector field $F_0\in L^{p,q}(\O)^N$ such that
\begin{equation*}V\phi=-\div(F_0)\hbox{ and }||F_0||_{L^{p,q}}\LEQ K_{1,N}(\O)||V||_{L^{r,q}}||\phi||_\infty,\end{equation*}
that is 
\begin{equation*}||F_0||_{L^{p,q}}\LEQ K_{1N}(\O)||V||_{L^{r,q}}||F||_{L^{p,q}(\O)}.\end{equation*}
Setting $F_1=F-F_0$, we can write (\ref{eq2}) as 
\begin{equation}\label{eq3bb}
-\Delta\phi=-\div(F_1-\ORA u\phi).
\end{equation}
But, we have $\ORA  u\phi\in L^{p,q}(\O)^N$ since  $\phi\in L^\infty(\O)$ according to the above Proposition~\ref{p4}.
Hence   
\begin{equation*}||\ORA u\phi||_{L^{p,q}(\O)^N}\LEQ ||\ORA u||_{L^{p,q}}||\phi||_{L^\infty}\LEQ K_N||F||_{L^{p,q}}||\ORA u||_{L^{p,q}}.\end{equation*}
We may apply the $W^1L^{p,q}$ result to (\ref{eq3bb}) (see 
\cite{s1, b1, a1, Ra2}) to deduce that
\begin{equation}\label{eq4}
||\nabla \phi||_{L^{p,q}}\LEQ K_p||F_1-\ORA u\phi||_{L^{p,q}}
\LEQ K_{pNq}(1+||\ORA u||_{L^{p,q}}+||V||_{L^{r,q}})||F||_{L^{p,q}}.
\end{equation}
For the general case, we consider $u_k\in\V$ such that  $u_k\to u$ strongly in $L^{p,q}(\O)^N$. Let $\phi_k$ be the solution  of equation (\ref{eq2}) where $ \phi$ is replaced by $\phi_k$
\begin{equation*}-\Delta\phi_k-\ORA u_k\nabla\phi_k+V{\phi_k}=T=-\div(F).\end{equation*}
The sequence $(\phi_k)_k$ satisfies \begin{equation*}||\phi_k||_{L^\infty}\LEQ K_N||F||_{L^{r,q}}\hbox{ and }||\phi_k||_{H^1_0}\LEQ ||T||_{H^{-1}},\end{equation*} and then  $(\phi_k)_k$ converges weakly in $H^1_0(\O)$ to $\phi$ the solution of \eqd. Since $\phi_k$ satisfies \eqq, we deduce that $\phi$ also satisfies \eqq and   \eqd.
\end{proof}

As an immediate consequence of the above result.
\begin{propo}\label{p6}\ \\
	Let $\ORA u$ and $\ORA F$ be in $L^{p,\infty}(\O)^N$ for some $p>N$. Then, the solution of \eqd satisfies \begin{equation*}\hbox{ $\phi\in C^{0,\a}(\OV\O)$ with $\a=1-\dfrac N p$.}\end{equation*}
\end{propo}
\begin{proof}
According to the Sobolev embedding (see  \cite{RakoRR}), we have 
\begin{equation*}W^1L^{p,\infty}(\O)\hookrightarrow C^{0,\a}(\OV\O),\hbox{ with }\a=1-\dfrac Np.\end{equation*}
\end{proof}
Now we shall consider the case of more general data $\ORA u$ and $V$.
\begin{propo}\label{p7}\ \\
	Assume that $\ORA u$ and $\ORA F$ are in $\bmor(\O)^N$ and $V$ is in $\bmo_r(V)$. Then the solution $\phi$ of the equation \eqd satisfies
	\begin{enumerate}
		\item $\ORA u\phi\in\bmor(\O)^N$
		\item $\nabla \phi\in\bmor(\O)^N$.
	\end{enumerate}
\end{propo}
\begin{proof}
Since $\bmor(\O)\hookrightarrow L^{p,q}(\O)$ for all $p>N$ and $q\in [1,+\infty]$, we  deduce from Proposition~\ref{p5}  and Proposition~\ref{p6}   that : 
\begin{equation*}\phi\in C^{0,\a}(\OV\O)\quad\forall\,\a\in\,[0,1\,[\hbox{ and }-\Delta\phi=-\div(\ORA{F_1}-\ORA u \phi),\end{equation*} where $\ORA F_1$ was defined in the proof of Proposition \ref{p5} \hbox{(see  equation (\ref{eq3bb}))}. 
From Stegenga multiplier's result, $\ORA u\phi\in\bmor(\O)^N$ whenever $\ORA u$ is in $\bmor(\O)^N$ \cite{ST,Tor}. Therefore $\ORA{F_1}-\ORA u\phi\in\bmor(\O)^N$. We may appeal to Campanato's result \cite{c1} to derive then that  $\nabla \phi\in\bmor(\O)^N$ and
\begin{equation*}
	||\nabla \phi||_\bmor\LEQ K\Big(||F||_\bmor+||\ORA u\,\phi||_\bmor+||F_0||_\bmor\Big).
\end{equation*}
\end{proof}

We shall end this paragraph by proving a $\WdLpq$-regularity result for the solutions of the dual problem \eqd which will lead to interesting conclusions for the direct problem~(\ref{eq1}).\\
For this, we shall use the following ADN constant 
\begin{equation}\label{eqADN} K^s_{pq}=\sup_{v\in H^1_0(\O)\cap \WdLpq}\dfrac{||v||_{\WdLpq}}{||v||_{L^{p,q}(\O)}+||\Delta v||_{L^{p,q}(\O)}},\end{equation}
which is  finite due to the well-known Agmon-Douglis-Nirenberg's regularity result combined with  the Marcinkiewicz interpolation Theorem.\\
We shall improve now the regularity obtained in Proposition \ref{p7}.
We consider $\eps_0>0$ (fixed) so that $K^s_{pq}\eps_0||\ORA u||_\Lpq\LEQ\frac12 $.
\begin{propo}\label{p8}{\bf ($\WdLpq$ regularity for $p>N$)}\\
	Let $\phi$ be the solution of \eqd when $T\in L^{p,q}(\O),\ p>N,\ q\in[1,+\infty]$. Assume, furthermore, that $\ORA u\in L^{p,q}(\O)^N$ and $V\in L^{p,q}(\O)$. Then
	\begin{equation*}\phi\in\WdLpq.\end{equation*}
	Moreover, there exist constants $c_{\eps_0},\ K_{pqN}>0$ such that
	\begin{equation*}||\phi||_\WdLpq\LEQ\dfrac{K_{pqN}c_{\eps_0}(1+||V||_{L^{p,q}}+||\ORA u||_{L^{p,q}(\O)})}{1-K^s_{pq}\eps_0||\ORA u||_\Lpq}||T||_\Lpq.\end{equation*}
\end{propo}
\begin{proof} 
We assume first that $\ORA u\in\V$. Arguing as in   Proposition \ref{p5}, since we can assume that $T=\div\ORA F$ for suitable $\ORA F$ we get that the solution $\phi$ of \eqd is in $W^1\Lpq$ and then \begin{equation*}-\Delta\phi=\ORA u\,\nabla\phi+T-V\phi\in\Lpq.\end{equation*}
By the Agmon-Douglis-Nirenberg regularity results and the Marcinkiewicz interpolation theorem we deduce that $\phi\in\WdLpq$. 
Moreover, since $p>N$ and $q\in[\,1,+\infty]$, we have the following continuous embeddings :
\begin{equation*}\WdLpq\hookrightarrow C^1(\OV\O)\hookrightarrow\Lpq.\end{equation*}
The first inclusion is compact so we may appeal to Lemma \ref{l2l} to derive that 
$\forall\,\eps>0$, there exists $c_\eps>0$ such that
\begin{equation}\label{eq5}
||\nabla\,\phi||_\infty\LEQ\eps||\phi||_\WdLpq+c_\eps||\phi||_\Lpq.
\end{equation}
From the equation satisfied by $\phi$, we have
\begin{equation}\label{eq6}
||\Delta\phi||_\Lpq\LEQ||\ORA u||_\Lpq||\nabla\phi||_\infty+||T||_\Lpq+||V||_{L^{p,q}}||\phi||_\infty,
\end{equation}
and using the ADN constant
\begin{equation}\label{eq7}
||\phi||_\WdLpq\LEQ K^s_{pq}\Big(||\phi||_\Lpq+||\Delta\phi||_\Lpq\Big).
\end{equation}
We combine those last three  equations and derive  that for any $\eps>0$
\begin{eqnarray}
||\phi||_\WdLpq(1-\eps K^s_{pq}||\ORA u||_\Lpq)
&\LEQ& K^s_{pq}||\phi||_\Lpq\Big(1+c_\eps||\ORA u||_\Lpq\Big)\nonumber\\
&&+K^s_{pq}||T||_\Lpq(1+||V||_{L^{p,q}})K_{2N}. \label{eq8}
\end{eqnarray}

Next, we consider $\ORA u_k\in\V$ such that $\ORA{u_k}\to\ORA u\in\OV\V$. 
Then, choosing $\eps= \eps_0>0$ such that $\DST\eps_0 K^s_{pq}\sup_k||\ORA{u_k}||_\Lpq\LEQ\dfrac12$, we deduce from relation (\ref{eq8}) that $\phi_k$   corresponding to the solution of \eqd,  that is $ -\Delta\phi_k-\ORA{u_k}\cdot\nabla\phi_k+V\phi_k=T\in\Lpq$, belongs to a bounded set of $\WdLpq$ when $k$ varies. 
Therefore, the strong limit   $\phi$ in $C^1(\OV\O)$ is the solution of \eqd and it satisfies also the relation (\ref{eq8}) for all $\eps\in]0,\eps_0\,]$.
From Proposition \ref{p3}, we have
\begin{equation}\label{eq9}
||\phi||_\Lpq\LEQ K_N(\O)||T||_\Lpq.
\end{equation}
Combining relations (\ref{eq8}) and (\ref{eq9}) with $\eps=\eps_0$, we derive the result. \end{proof}
\ \\
The case where $p=N$ can also be treated in the same way provided that the norm of $\ORA u$ in $L^{N,1}(\O)$ is small enough in the sense that 
\begin{eqnarray}\label{eq10}
||\ORA u||_{L^{N,1}(\O)}&\LEQ&\theta K^{s0}_{N1}\hbox{ for some }\theta\in[\,0,1\,[, \\ K^{s0}_{N1}&=&K^s_{N1}\sup_{\phi\in H^1_0(\O)\cap W^2L^{N,1}(\O)}\dfrac{||\nabla\phi||_\infty}{||\phi||_{W^2L^{N,1}}}.
\end{eqnarray}

\begin{propo}\label{p9}{\bf (Regularity in $\WdLNu$).}\\
	Let $\phi$ be the solution of \eqd when $T\in\LNu$, $V\in L^{N,1}(\O)$. Assume that $\ORA u$ satisfies relation (\ref{eq10}). Then $\phi\in\WdLNu$. Moreover, there exists a constant $K'_N(\O)$ (independent of $\ORA u$) such that
	\begin{equation*}||\phi||_\WdLNu\LEQ\dfrac{K'_N(\O)(1+||V||_{L^{N,1}})}{1-K^{s0}_{N1}||\ORA u||_{L^{N,1}}}||T||_\LNu.\end{equation*}
\end{propo}
\begin{proof}
The proof follows the same argument as for the proof of Proposition \ref{p8}.
Nevertheless, the embedding $W^1L^{N,1}\subset C(\OV\O)$ is not compact and this explains the condition~(\ref{eq10}).
\end{proof}

There are many other spaces between the space $L^{p,1}(\O)$ and $\LNu$ for which we can obtain a regularity result for the second derivatives of $\phi$.

Here we want only to consider the space $\Lambda=(L^N(\ln L)^{\frac\b N})^N$ for $\b>N-1$.

Indeed this space is included in $\LNu$ and contains $L^p(\O)$ for all  $p>N$.

\begin{theo}\label{t4}{\bf (Regularity in $W^2L^N(\O)$).}\\
	Let $T$ and $V$ be  in $L^N(\O),\ \ORA u\in \Lambda$, $\div(\ORA u)=0$ and $\ORA u\cdot\ORA n=0$ on $\p\O$. Then the unique solution  $\phi$ of \eqd belongs to $W^2L^N(\O)$ and choosing $\eps>0$ such that $\eps||\ORA u||_\Lambda\LEQ\dfrac12$, there exists a constant $K_\eps>0$ such that
	\begin{equation*}||\phi||_{W^2L^N(\O)}\LEQ\dfrac{K_\eps(1+||\ORA u||_\Lambda+||V||_{L^N})}{1-\eps||\ORA u||_\Lambda}||T||_{L^N(\O)}.\end{equation*}
\end{theo}

The proof   firstly depends on the following Trudinger's type embedding :
\begin{lem}\label{l}{\bf (Trudinger's embedding)}\ \\
	We have
	\begin{equation*}W_0^1L^N(\O)\hookrightarrow L^{\frac1{N'}}_\exp(\O).\end{equation*}
	Moreover,  for all $v\in W^1_0L^N(\O)$
	\begin{equation*}\sup_{t\LEQ|\O|}\dfrac{|v|_*(t)}{\left(1+\ln\dfrac{|\O|}t\right)^{\frac1{N'}}}\LEQ K_0||\nabla v||_{L^N(\O)},\hbox{ with } K_0=\dfrac1{N\a^{\frac1N}_N}.\end{equation*}
\end{lem}
\begin{proof} 
According to the pointwise Sobolev inequality for the relative rearrangement, we have for $u=|v|$ (see Theorem \ref{t10})
\begin{equation}\label{eq200}
-u'_*(s)\LEQ\dfrac{s^{\frac1N-1}}{N\a_N^{\frac1N}}|\nabla u|_{*u}(s).
\end{equation}
We integrate this formula from $t$ to $|\O|$ knowing that $u_*(|\O|)=0$, and using the H\"older inequality, we get
\begin{equation}\label{eq201}
u_*(t)\LEQ\dfrac1{N\a_N^{\frac1N}}\int_t^{|\O|}s^{\frac1N-1}|\nabla u|_{*u}(s)ds\LEQ\dfrac1{N\a^{\frac1N}_N}\left(\ln\dfrac{|\O|}t\right)^{\frac1{N'}}||\,|\nabla u|_{*u}||_{L^N}.
\end{equation}
Therefore from (\ref{eq201}), implies using Theorem \ref{t1t1}

\begin{equation*}\sup_{t\LEQ|\O|}\dfrac{u_*(t)}{\left(1+\ln\dfrac{|\O|}t\right)^{\frac1{N'}}}
\LEQ\dfrac1{N\a_N^{\frac1N}}||\,|\nabla u|_{*u}||_{L^N}
\LEQ\dfrac1{N\a_N^{\frac1N}}||\nabla u||_{L^N}.\end{equation*}
\end{proof}

The key result for the proof of Theorem \ref{t4} is the following compactness inclusion~:
\begin{theo}\label{t2.2}{\bf (Compact inclusion for $W^1_0L^N(\O)$).}\\
	$W^1_0L^N(\O)$ is compactly embedded in $L^\a_\exp(\O)$ for $\a>\dfrac1{N'}$.
\end{theo}
\begin{proof} 
Let $(u_n)_n$ be a bounded sequence in $W^1_0L^N(\O)$. We may assume that $u_n\rightharpoonup u$ in $W^1_0L^N(\O)$-weakly and almost everywhere in $\O$. Let $c=\DST\Max_n||u_n-u||_{L^{\frac1{N'}}_\exp}<+\infty$.\\
For $\eps>0$, there exists $\d>0$ such that
\begin{equation*}\dfrac c{\left(1+\ln\dfrac{|\O|}t\right)^{\a-\frac1{N'}}}\LEQ \eps \hbox{ for all }t\LEQ\d.\end{equation*}
Therefore, we have :
if $t\LEQ\d$
\begin{equation*}\dfrac{|u_n-u|_*(t)}{\left(1+\ln \dfrac{|\O|}t\right)^\a}\LEQ\dfrac c{\left(1+\ln\dfrac{|\O|}t\right)^{\a-\frac1{N'}}}\LEQ \eps ;\end{equation*}
if $t>\d$ then, since $|u_n-u|_*$ is nonincreasing 
\begin{equation*} {|u_n-u|_*(t)} \LEQ\dfrac1{\d }\int_0^\d|u_n-u|_*(s)ds,\end{equation*}
so that
\begin{equation*}\sup_{t\GEQ \d}\dfrac{|u_n-u|_*(t)}{\left(1+\ln\dfrac{|\O|}t\right)^\a}\LEQ\dfrac1{\d }\int_0^\d|u_n-u|_*(s)ds.\end{equation*}
The right hand side of this inequality tends to zero as $n$ goes to infinity. Hence, for $n\GEQ n_\eps$ with  $n_\eps$ large enough
\begin{equation*}
	\sup_{0<t<|\O|}\dfrac{|u_n-u|_*(t)}{\left(1+\ln\dfrac{|\O|}t\right)^\a}\LEQ \eps.
\end{equation*}
\end{proof}

As a corollary of the above theorem, since $W^2L^N\cap W^1_0L^N\hookrightarrow W^1_0L^\a_\exp\hookrightarrow L^N$, we have:
\begin{cor}\label{c2.t3}{\bf (of Theorem \ref{t2.2})}\\
	Let $\a>\frac1{N'}$. Then, for every $\eps>0$, there exists $c_\eps>0$ such that $\forall\,v\in W^2L^N(\O)\!\cap\,H^1_0(\O)$
	\begin{equation*}||\nabla v||_{L^\a_\exp}\LEQ\eps||\Delta v||_{L^N}+c_\eps||v||_{L^N}.\end{equation*}
\end{cor}
\begin{proof}
We  use the equivalence of norms $||v||_{W^2L^N(\O)\cap H^1_0}\!\!\equiv\!\!||\Delta v||_{L^N}\!+\!||v||_{L^N}$ and apply Lemma~\ref{l2l} with
\begin{equation*}Y=W^1_0L^\a_\exp(\O),\quad X=W^2L^N(\O)\cap H^1_0(\O),\quad Z=L^N(\O).\end{equation*}
\end{proof}

\begin{proof}[Proof of Theorem \ref{t4}]
We first assume that $\ORA u\in\V$, and $T\in L^\infty(\O)$. Then, the unique solution $\phi$ of \eqd satisfies 
\begin{eqnarray}
||\Delta\phi||_{L^N}&\LEQ&||T||_{L^N}+||\ORA u\cdot\nabla \phi||_{L^N}+||V||_{L^N}||\phi||_\infty\nonumber\\
&\LEQ&K_N(1+||V||_N)||T||_N+||\ORA u\cdot\nabla\phi||_{L^N}\label{eq11}.
\end{eqnarray}
We have
\begin{equation*}||\ORA u\cdot\nabla\phi||^N_{L^N}
\LEQ\int_{\O_*}|\ORA u|^N_*|\nabla\phi|^N_*dt\LEQ \sup_{t\in\O_*}\dfrac{|\nabla \phi|_*^N(t)}{\left(1+\ln\dfrac{|\O|}t\right)^\b}\int_{\O_*}|\ORA u|^N_*(t)\left(1+\ln\dfrac{|\O|}t\right)^\b dt,\end{equation*}
which implies
\begin{equation}\label{eq12}
||\ORA u\nabla \phi||_{L^N}\LEQ
||\nabla \phi||_{L_\exp^\a}||\ORA u||_\Lambda\hbox{ with }\a=\dfrac\b N>\dfrac1{N'}.
\end{equation}
Let $\eps>0$ be fixed. There exists $c_\eps>0$ such that
\begin{equation*}||\ORA u\cdot\nabla \phi||_{L^N}\LEQ(\eps||\Delta\phi||_{L^N}+c_\eps||\phi||_{L^N})||\ORA u||_\Lambda\end{equation*}
(see  Corollary \ref{c2.t3} of Theorem \ref{t2.2}). 
Combining this  with relation (\ref{eq11}), we have $\forall\,\eps>0,\ \exists\,c_\eps^1>0$
\begin{equation}\label{eq13}
||\Delta\phi||_{L^N}(1-\eps||\ORA u||_\Lambda)\LEQ c^1_\eps(1+||\ORA u||_\Lambda
+||V||_{L^N})||T||_{L^N}.
\end{equation}
Secondly, we consider $T\in L^N(\O)$ and $\ORA u\in\OV\V$. 
There exist $\ORA{ u_k}\in\V$ such that $\ORA u_k\to\ORA u$  strongly in $\Lambda$ 
and $T_k\in L^\infty(\O)$ with \begin{equation*}||T_k||_{L^N}\LEQ ||T||_{L^N}.\end{equation*}
Then from relation (\ref{eq13}), the solution $\phi_k$ of \eqd satisfies
\begin{equation}\label{eq14}
||\Delta \phi_k||_{L^N}(1-\eps||\ORA {u_k}||_\Lambda)\LEQ c_\eps^1(1+||\ORA u_k||_\Lambda+||V||_{L^N})||T||_{L^N}.
\end{equation}
We choose $\DST\eps_0>0$ such that \begin{equation*}\eps_0\sup_k||u_k||_\Lambda\LEQ\dfrac12.\end{equation*}
Then $\phi_k$ remains in a bounded set of $W^2L^N(\O)\cap H^1_0(\O)$. So it converges   to $\phi$ weakly in $W^2L^N(\O)\cap H^1_0(\O)$ and we have
\begin{equation}\label{eq15}
||\Delta\phi||_{L^N}(1-\eps_0||\ORA u||_\Lambda)\LEQ c_{\eps_0}^1(1+||\ORA u||_\Lambda+||V||_{L^N})||T||_{L^N},
\end{equation}
and
\begin{equation*}
||\phi||_{L^N}\LEQ |\O|^{\frac1N} ||\phi||_\infty\LEQ K_N(\O)||T||_{L^N(\O)}
\end{equation*}
(according to Proposition \ref{p3}). 
This gives the results.
\end{proof}


\section{Very weak solutions of problem (\ref{eq1}) with and without the Dirichlet boundary condition.}
We now want to apply all those regularity results to the study of equation (\ref{eq1}). We first start with  some definitions of the weak solution associated with (\ref{eq1}).

\subsection{Existence and regularity of the very weak solution for a locally integrable potential $V\GEQ0$}
We start by considering the existence of very weak solutions of equation (\ref{eq1}) with the Dirichlet boundary condition \eqd when the potential $V$ is a nonnegative locally integrable function.\\
We can use the definition of very weak solution (see Definition \ref{debdef}).

\begin{theo}\label{t2.3}\ \\
	Let $f\in L^1(\O;\d)$. Let $\ORA u$ be in $L^{p,1}(\O)^N$ with $\div(\ORA u)=0$ in $\D'(\O)$, 
	$\ORA u\cdot\ORA n=0$ on $\p\O$. Furthermore, assume that either $p>N$ or $p=N$ and 
	$||\ORA u||_{L^{N,1}}< K^{s0}_{N1}$ (see  (\ref{eq10})). Then,  there exists a  very weak solution $\omega$ in the sense of (\ref{eq1XV}), which is unique, if   $V\in L^{p,1}(\O)$.
\end{theo}
\begin{rem}\ \\
	In section 4.2,  we shall discuss the uniqueness of the v.w.s when  $V\notin L^{N,1}(\O)$.
\end{rem}
\begin{proof} 
First, we assume  that $f\GEQ0$. 
Let $u_j\in\V$ be such that ${\ORA u}_j\to\ORA u$ strongly in $L^{p,1}(\O)^N$ and $f_j\in L^\infty(\O)$ such that $0\LEQ f_j(x)\LEQ f(x)$ a.e and $f_j(x)\to f(x)$ a.e. According to Proposition \ref{p3b}, Proposition \ref{p8} or Proposition \ref{p9}, there exists a unique function $\omega_j\GEQ0$ such that 
\begin{equation}\label{eq17}
\begin{cases}
-\Delta\omega_j+{\ORA u}_j\cdot\nabla \omega_j+V_j\omega_j=f_j,\\
\omega_j\in H^1_0(\O)\cap W^2L^{p,1}(\O),
\end{cases}
\end{equation}
which is equivalent to saying that 
\begin{equation}\label{eq18}
\begin{cases}
\DST\int_\O\omega_j\Big[-\Delta\phi-{{\ORA u}_j}\cdot\nabla \phi\Big]dx=\int_\O f_j\phi\, dx
-\int_\O V_j\omega_j\phi dx,\\
\forall\,\phi\in W^2L^{p,1}(\O)\cap H^1_0(\O).
\end{cases}
\end{equation}
We argue as in \cite{DR1, DRJFA, Ra2}. Let $E$ be a measurable subset of $\O$ and $\chi_E$ its characteristic function. Then, there exists a non negative function $\phi_j\in W^2L^m(\O),\ \forall\, m<+\infty$, satisfying
\begin{equation}\label{eq19}
\begin{cases}
-\Delta\phi_j-{{\ORA u}_j}\nabla \phi_j=\chi_E\hbox{ in }\O,\\
\phi_j=0\hbox{ on }\p\O.
\end{cases}
\end{equation}
We consider a small number $\eps>0$ such $\eps\DST\sup_j||{\ORA u}_j||_{L^{N,1}}\LEQ\dfrac12$. Therefore, we have 
\begin{equation*}||\phi_j||_{W^2L^{N,1}}\LEQ K_0||\chi_E||_{L^{N,1}}\LEQ K_1|E|^{\frac1N}.\end{equation*}
Thus
\begin{eqnarray}\label{eq20}
\int_E\omega_jdx=\int_\O\omega_j\big[-\Delta\phi_j-\ORA {u_j}\nabla \phi_j\big]dx
&\LEQ&\int_\O f_j\phi_j\LEQ K_1\left(\int_\O|f_j|\d\right)||\phi_j||_{W^2L^{N,1}}\nonumber\\
&\LEQ& K_0|E|^{\frac1N}
\int_\O|f_j|\d dx.\end{eqnarray}
By the Hardy-Littlewood  property we conclude that
\begin{equation}\label{eq600}
\sup_{t\LEQ|\O|}t^{\frac 1{N'}}|\omega_j|_{**}(t)\LEQ K_0\int_\O|f_j|\d dx\LEQ K_0\int_\O|f|\d dx.
\end{equation}
Moreover, choosing  $\phi=\f_1$ as the test function  with $-\Delta\f_1=\lambda_1\f_1$, and $\f_1=0$ on $\partial\O$, we have
\begin{eqnarray*}
	\lambda_1\int_\O\omega_j\f_1dx+\int_\O V_j\omega_j\f_1dx&\LEQ& ||\nabla\f_1||_\infty\,||\omega_j||_{L^{N',\infty}}||{\ORA u}_j||_{L^{N,1}}+c\int_\O|f_j|\d dx \\
	&\LEQ&c\big(1+||{\ORA u}_j||_{L^{N,1}}\big)\int_\O|f_j|\d dx,
\end{eqnarray*}
for a suitable constant $c>0$.
Thus $V_j\omega_j$ remains in a bounded set of $L^1(\O;\d)$ and
\begin{equation}\label{eq601}
\int_\O V_j\omega_j\d dx\LEQ c\big(1+||{\ORA u}_j||_{L^{N,1}}\big)\int_\O|f_j|\d dx.
\end{equation}
If $f$ has a constant sign, we write  $f_j=f_{j+}-f_{j-}$ with $f_{j+}=\max(f_j,0)\GEQ0$.\\
Denoting by $\omega^+_j$ the v.w.s. associated to $f_{j+}$ and  by $\omega_{j}^-$ the one associated to $f_{j-}$,we see that $\omega_j=\omega_j^+-\omega_j^-$ satisfies (\ref{eq18}) and  we have also the estimates (\ref{eq600}) and (\ref{eq601}).\\
In particular, since $|\omega_j|\LEQ\omega_j^++\omega_j^-$  
\begin{equation}\label{eq3900}
\int_\O V_j|\omega_j|\d dx\LEQ c\big(1+||u_j||_{L^{N,1}}\big)\int_\Omega|f_j|\d dx.
\end{equation}

We conclude that $(\omega_j)_j$ converges weak-* to $\omega$ in $L^{N',\infty}(\O)=\big(L^{N,1}(\O)\big)^*$.
To obtain 
a strong convergence, we need a local estimate of the gradient.
For that purpose, we shall  prove  the boundedness of $\omega_j$  in the Lorentz-Sobolev weighted space $W^1L^{1+\frac1N,\infty}(\O;\d)$. For this, we shall need the following result due to Philippe  B\'enilan  and co-authors whose proof can be found in \cite{BBGG} Lemma 4.2, with generalization  in \cite{RabookLinear}.

\begin{propo}\label{p11}\ \\
	Let $v\in L^1(\O,\d^\a),$ and $ \a\in[\,0,1\,]$. Assume that there exists a constant $c_0>0$ such that  for all $k>0$
	\begin{equation*}T_k(v):=\min(|v|;k)\ \sign(v)\in W^1L^2(\O,\d^\a),\end{equation*}
	and 
	\begin{equation}\label{eq24bis}
	\int_\O|\nabla T_k(v)|^2\d^\a dx+\int_\O|T_k(v)|^2\d^\a dx\LEQ c_0 k.
	\end{equation}
	Then, there exists a constant $c$, depending continuously on $c_0>0$, such that for all $\lambda>0$
	\begin{equation*}\int_{\{x:|\nabla v|(x)>\lambda\}}\d^\a(x)dx\LEQ\dfrac c{\lambda^{1+\frac1{N+\a-1}}}.\end{equation*}
	
	In particular, if $v_j$ is a sequence converging weakly in $L^1(\O)$ to a function $v$, satisfying the inequality (\ref{eq24bis})
	\begin{equation*}\int_\O|\nabla T_k(v_j)|^2\d^\a dx\LEQ c_0k\qquad\qquad\forall j,\ \forall k,\end{equation*}
	then $v_j$ converges to $v$ weakly in $W^{1,q}(\O')$ for all  $q\in\left[1,\dfrac{N+\alpha}{N+\a-1}\right[$ and all $\O'\subset\!\subset\O$, with a subsequence, $v_j(x)\to v(x)$ a.e. in $\O$.
\end{propo}

We first need to prove the following a priori estimate :
\begin{propo}\label{p12}\ \\
	Let $\omega_j$ be the solution of (\ref{eq17}),  $\omega$ its weak limit in $L^{N',\infty}(\O)$. %
	Under the same assumptions as for Theorem \ref{t4}, there exists a constant $c_0>0$ such that:
	\begin{equation*}\int_\O\big|\nabla T_k(\omega_j)\big|^2\d dx+\int_\O\big|\nabla T_k(\omega)\big|^2\d dx\LEQ c_0 k\quad \forall\,k>0,\ \forall\,j.\end{equation*}
\end{propo}
 \begin{proof}
Let $\f_1$ be the first eigenvalue of the Dirichlet problem $-\Delta\f_1=\lambda_1\f_1$ in $\O$, $\f_1=0$ on $\p\O$. Then, there exist constants such that $c_1\d(x)\LEQ\f_1(x)\LEQ c_2\d(x)\quad\forall\,x\in\O.$
We consider the approximate problem given in equation (\ref{eq17}) say 
\begin{equation*}\begin{cases}
-\Delta \omega_j+{\ORA u}_j\cdot\nabla\omega_j+V_j\omega_j=f_j,\\
\omega_j\in W^{1,1}_0(\O)\cap W^2L^{p,1}(\O),
\end{cases}\end{equation*}
with $|f_j(x)|\LEQ |f(x)|,\ f_j\to f\hbox{ a.e}, {{\ORA u}_j}\to\ORA u\hbox{ in } L^{p,1}(\O)^N-$strongly and $\omega_j\to\omega$ weakly-* in $L^{N',\infty}(\O)$.\\
For $k>1$, we choose $T_k(\omega_j)\f_1$ as a test function; then $V_j\omega_jT_k(\omega_j)\f_1\GEQ0$ and we derive after some integrations by parts :
\begin{equation}\label{eq23}
\int_\O\!\!|\nabla T_k(\omega_j)|^2\f_1dx+\lambda_1\!\!\int_\O\!\!\f_1\!\!\left(\!\!\int_0^{\omega_j}\!\!\!\!\!\!\!T_k(\s)d\s\right)dx
-\int_\O{\ORA u}_j\cdot\nabla\f_1\int_0^{\omega_j}\!\!\!\!\!T_k(\s)d\s dx\!\LEQ\!c_2 k\int_\O\!|f|\d dx. 
\end{equation}
This relation implies:
\begin{equation}\label{eq24}\ \\
\int_\O|\nabla T_k(\omega_j)|^2\d(x)\LEQ c_3k\int_\O|\omega_j|\d dx+c_2 k\int_\O|f|\d dx+c_3k\int_\O|{\ORA u}_j|\,|\omega_j|dx.
\end{equation}
By the H\"older inequality
\begin{equation}\label{eq25}\ \\
\int_\O|{\ORA u}_j|\,|\omega_j|dx
\LEQ c_4||{\ORA u}_j||_{L^{N,1}}\cdot||\omega_j||_{L^{N',\infty}}
\LEQ c_4||\ORA u_j||\int_\O|f|\d dx.
\end{equation}
From relation  (\ref{eq24}) and  (\ref{eq25}), we then have :
\begin{equation}\label{eq26}
\int_\O|\nabla T_k(\omega_j)|^2\d(x)dx\LEQ c_5(1+||{\ORA u}_j||_{L^{N,1}})\left(\int_\O|f|\d dx\right)k.
\end{equation}
Letting $j\to+\infty$, we deduce from  (\ref{eq26}) and Proposition \ref{p11} :
\begin{equation*}\int_\O|\nabla T_k(\omega)|^2\d(x)dx\LEQ c_0 k\hbox{ with } c_0=c_5(1+||\ORA u||_{L^{N,1}})\int_\O|f|\d dx.\end{equation*}
Then the $L^{N',\infty}$-regularity of $\omega$ implies 
\begin{equation*}\int_\O|T_k(\omega)|^2\d dx\LEQ c_0k\int_\O|\omega|dx.\end{equation*}
\end{proof}

\begin{cor}[of Propositions \ref{p11} and  \ref{p12}] \ \\
	Let $\omega$  be as in the proof of the previous proposition. Then, 
	there exists a constant $c_6>0$ such that
	\begin{equation*}||\nabla \omega||_{L^{1+\frac1N,\infty}(\O;\d)}\LEQ c_6\int_\O|f(x)|\d(x) dx.\end{equation*}
	In particular, we have, for all $q<1+\dfrac1N,$
	\begin{equation*}\int_\O|\nabla \omega|^q\d(x)dx\LEQ c_q\int_\O|f(x)|\d(x)dx.\end{equation*}
\end{cor}
To pass to the limit in (\ref{eq17}), we argue as in \cite{DiRa2} p. 1041. We emphasize the main differences due to the additional term $\ORA u\cdot\nabla\omega$.\\ 

Let us note that  by the above Proposition \ref{p8}, we have (for a subsequence still denoted as $(\omega_j)_j$) that
\begin{enumerate}
	\item $\omega_j(x)\to \omega(x) $ a.e. (and thus $V_j\omega_j\to V\omega$ a.e. in $\O$).
	\item$\omega_j\rightharpoonup\omega$ weakly in $W^{1,q}(\O;\d),\ \forall\,q<1+\dfrac1N$.
	\item $\omega_j\to\omega$ strongly in $L^r(\O),\hbox{ for any } r<N'$.
\end{enumerate}
In particular, we deduce from  the above statement 1., relation (\ref{eq3900}) and Fatou's lemma
\begin{lem}\ \\
	Under the assumptions of Theorem \ref{t2.3} and Proposition \ref{p12} one has
	\begin{equation*}\int_\O V|\omega|\d dx\LEQ c\Big(1+||u||_{L^{N,1}}\Big)\int_\O|f|\d dx.\end{equation*}
\end{lem}
\begin{lem}\label{l676}\ \\
	Under the assumptions of Theorem \ref{t2.3} and Proposition \ref{p12} one has
	\begin{equation*}\lim_{j\to+\infty}\int_\O|{\ORA u}_j\omega_j-\ORA u\omega|dx=0.\end{equation*}
\end{lem}
 \begin{proof}
Since ${\ORA u}_j\to\ORA u$ in $L^{N,1}(\O)$, and a.e. in $\O$, we have
\begin{equation*}\lim_{j\to+\infty}{\ORA u}_j(x)\omega(x)=\ORA u(x)\omega(x)\ a.e.\end{equation*}
It is enough to show that $({\ORA u}_j\omega_j)_j$ satisfies  Vitali's condition :  $\forall\,\eps>0$ $\exists \eta>0$ such that  if $E\subset\O $ is measurable with $|E|\LEQ\eta$ then
\begin{equation*}\limsup_{j\to+\infty}\int_E|{\ORA u}_j\omega_j|dx\LEQ\eps.\end{equation*}
But  from H\"older's inequality we have 
\begin{equation*}\int_E|{\ORA u}_j\omega_j|dx\LEQ||{\ORA u}_j||_{L^{N,1}(E)}||\omega_j||_{L^{N',\infty}(\O)}\LEQ c||{\ORA u}_j||_{L^{N,1}(E)},\end{equation*}
so that
\begin{equation*}\limsup_{j\to+\infty}\int_E|{\ORA u}_j\omega_j|dx\LEQ c||\ORA u||_{L^{N,1}(E)}.\end{equation*}
Since \begin{equation*}||\ORA{u}||_{L^{N,1}(E)}\xrightarrow[|E|\to0]{}0,\end{equation*}
we derive that it satisfies the Vitali condition. Therefore, we have proved the lemma.~ \end{proof}

Then we have the following result analogous to   Lemma 2.3 of \cite{DiRa2}.
\begin{lem}\ \\
	We assume that $V \in L^1_{loc}(\O),$ and $ V\GEQ0$. 
	Then
	\begin{equation*}\hbox{$V_j\omega_j\d \rightharpoonup V\omega\d$  \hbox{ weakly}  in $L^1_{loc}(\O)$}.\end{equation*}
	Furthermore, if $V\in L^1(\O;\d)$,  then \begin{equation*}V_j\omega_j\d\rightharpoonup V\omega\d \hbox{ weakly in }L^1(\O).\end{equation*}
\end{lem}
 \begin{proof}
Let $t\in\R_+$. Consider a sequence of functions $\gamma_m$ in $C^1(\R)\cap W^{1,\infty}(\R)$ such that
\begin{eqnarray*}
	\gamma'_m\GEQ0    &&\ \forall\,s\in\R,\\
	\gamma_m(s)&\to&-1\hbox{\ for\ }s<-t\hbox{ as\ }m\to+\infty,\\
	\gamma_m(s)&\to&1\hbox{\ for\ }s>t\hbox{ as\ }m\to+\infty,\\
	\gamma_m(s)&=&0\hbox{\ on\ }-t\LEQ s\LEQ t,
\end{eqnarray*}
and let $\f_1\in C^2(\OV\O)$ with $-\Delta\f_1=\lambda_1\f_1$ in $\O$, $\f_1=0$ on $\p\O$, $\lambda_1>0$.\\
Taking  $\f_1\gamma_m(\omega_j)$ as a test function in relation (\ref{eq17}) we get
\begin{eqnarray}\label{eq701}
\int_\O\nabla\omega_j\cdot\nabla\big(\f_1\gamma_m(\omega_j)\big)
+\int_\O V_j\omega_j\f_1\big(\gamma_m(\omega_j)\big)dx
&+&\int_\O{\ORA u}_j\cdot\nabla\omega_j\gamma_m(\omega_j)\f_1dx\nonumber\\
&=&\int_\O f_j\gamma_m(\omega_j)\f_1dx.
\end{eqnarray}
We write $\DST \nabla \omega_j\gamma_m(\omega_j)=\nabla\left[\int_0^{\omega_j}\gamma_m(\s)dx\right]$ so that
\begin{eqnarray*}
	\int_\O({\ORA u}_j\cdot\nabla\omega_j)\gamma_m(\omega_j)\f_1dx
	&=&-\int_\O\div(u_j\f_1)\int_0^{\omega_j}\gamma_m(\s)d\s dx\\
	&=&-\int_\O{\ORA u}_j\nabla\f_1\left(\int_0^{\omega_j}\gamma_m(\s)d\s\right)dx.
\end{eqnarray*}
As $m\to+\infty$, treating the remaining terms in (\ref{eq701}) as in \cite{DiRa2}, we derive
\begin{equation}\label{eq702}
\int_{|\omega_j|>t}\!\!V_j|\omega_j|\d dx\LEQ
c\left[\int_{|\omega_j|\GEQ t}\!\!|f|\d dx+\int_{|\omega_j|\GEQ t}\!\!|\omega_j|\d dx+\int_{|\omega_j|\GEQ t}\!\!|{\ORA u}_j|\,|\omega_j|dx
\right].
\end{equation}
This relation proves that $V_j\omega_j\d$ remains in a bounded set of $L^1(\O)$ but also that the set $\Big\{ V_j|\omega_j|\d,\ j\in\N\Big\}$ is x compact for the $\s(L^1;L^\infty)$-topology, so we may appeal to the Dunford-Pettis to conclude. 
Indeed, let us set
\begin{equation*}\Gamma_j(t):=\int_{|\omega_j|\GEQ t}|f(x)|\d(x)dx+\int_{|\omega_j|\GEQ t}|\omega_j|\d dx+\int_{|\omega_j|\GEQ t}|{\ORA u}_j\omega_j|dx.\end{equation*}
 For a.e. $t>0$,
\begin{equation*}\lim_{j\to+\infty}\Gamma_j(t)=\Gamma(t)=\int_{|\omega|>t}|f(x)|\d(x) dx+\int_{|\omega|>t}|\omega|\d dx+\int_{|\omega|>t}|\ORA u\omega|dx,\end{equation*}
and  
\begin{equation*}\Big|\big\{|\omega|>t\big\}\Big|+\sup_j\Big|\big\{|\omega_j|>t\big\}\Big|\LEQ\dfrac{constant}t\xrightarrow[t\to+\infty]{}0,\end{equation*}
we deduce that for any $\eps>0$, there exists $t_\eps>0$ such that, for all $j\in\N$, \begin{equation*}\Gamma_j(t_\eps)\LEQ \eps.\end{equation*}
Let $\O_0\subset\O$ such that $V\d\in L^1(\O_0) $ (thus $\O_0\neq\O$ if $V$ is only locally integrable). Then by the Lebesgue convergence dominate theorem for a.e. $t$,
\begin{equation*}\lim_{j\to+\infty}\int_{\O_0}\Big|\chi_{|\omega_j|\LEQ t}(x)V_j\omega_j(x)-\chi_{\{|\omega|\LEQ t\}}(x)V(x)\omega(x)\Big|\d(x)dx=0,\end{equation*}
since \begin{equation*}\lim_{|A|\to0}\int_A V|\omega|\d dx=0\quad\Big(V\omega\d\in L^1(\O)\Big).\end{equation*}
Therefore there exists $\eta>0$ such that if $A\subset\O_0,\ |A|\LEQ\eta$, then for all $j\in\N$,
\begin{equation*}\int_{A\cap\{|\omega_j|\LEQ t_\eps\}}V_j|\omega_j|\d dx\LEQ\eps.\end{equation*}
Hence, for all $j\in\N$, all $A\subset\O_0,$ with $|A|\LEQ\eta$
\begin{equation*}\int_AV_j|\omega_j|\f dx\LEQ \Gamma_j(t_\eps)+\int_AV_j|\omega_j|\d dx\LEQ2\eps.\end{equation*}
This conclude the proof of Lemma 7.
\  \end{proof}

The passage to the limit, we will distinguish two different cases :\begin{enumerate}
	\item\underline{Case $V\in L^1(\O;\d)$} 
	For all $\phi\in C^2(\OV\O),\ \phi=0$,  we have
	\begin{equation}\label{eq760}
	\lim_j\int_\O V_j\omega_j\phi dx=\int_\O V\omega\phi dx\end{equation}
	(since  $\dfrac\phi\d\in L^\infty(\O)$ and $V_j\omega_j\d$  converges to $V\omega\d$ for  $\s(L^1;L^\infty)$ topology). 
	Therefore, since
	\begin{equation}\label{eq761}
	-\int_\O \omega_j\Delta\phi dx-\int_\O{\ORA u}_j\omega_j\nabla\phi dx+\int_\O V_j\omega_j\phi dx=\int_\O f_j\phi dx,
	\end{equation}
	we let $j\to+\infty$ to deduce that $\omega$ is a v.w.s. using Lemma \ref{l676} and the convergences of $\omega_j$.
	\item \underline{Case $V\in L^1_{loc}(\O)$} 
	We consider $\phi\in W^2L^{N,1}(\O)$ with support $\phi$ be a compact in $\O$. Then the same argument holds since $V_j\omega_j\d$ tends to $V\omega\d$ weakly in $L^1_{loc}(\O)$. Then (\ref{eq760}) and (\ref{eq761}) hold true
	\begin{equation}\label{eq762}
	\begin{cases}
	\DST\int_\O\omega\big[-\Delta\phi-\ORA u\nabla\phi+V\phi\big]dx=\int_\O f\phi dx,\\
	\forall\,\phi\in W^2L^{N,1}(\O),\hbox{ support$(\phi)$ compact in $\O$}.
	\end{cases}
	\end{equation}
\end{enumerate}
If $V\in L^{p,1}(\O),$ the solution is unique. Indeed, if we denote by $\omega$ the difference of two solutions then
\begin{equation*}\int_\O\Big[-\Delta\f-\ORA u\nabla\f+V \f\Big]\omega dx=0\quad\forall\,\f\in C^2(\OV\O),\ \f=0\hbox{ on }\d\O.\end{equation*}

Let us consider the function $\phi$ solution of
\begin{equation}\label{eq21}
\begin{cases}
-\Delta\phi-\ORA u\nabla\phi+V\phi=\sign(\omega),\\
\phi\in H^1_0(\O).
\end{cases}
\end{equation}
Then $\phi\in W^2L^{N,1}\O\hookrightarrow C^1(\OV\O)$ for $V\in L^{N,1}(\O)$. Thus
\begin{equation}\label{eq22}
\int_\O\omega\big[-\Delta\phi-\ORA u\nabla\phi+V\phi\big] dx=0,
\end{equation}
since $\Big\{\f\in C^2(\OV\O):\f=0\hbox{ on }\p\O\Big\}$ is dense in $W^2L^{N,1}(\O)\cap H^1_0(\O)$. Combining the relations (\ref{eq21}) and  (\ref{eq22}) we find :
\begin{equation*}\int_\O|\omega|dx=0\quad i.e.\ \omega\equiv 0.\end{equation*}

\  \end{proof}

\subsection{A result of uniqueness of solution when the potential is bounded from below by $  c\,\d^{-r},r>2$}\label{ss322}

The purpose of this section is to show the following uniqueness result.
\begin{theo}\label{t222}\ \\
	Assume that $V$ is locally integrable $V \GEQ0$, and such that  
	\begin{equation*}\exists\,c>0,\ V(x)\GEQ c\d(x)^{-r},\hbox{ in a neighborhood  $U$ of the boundary, with }\ r>2.\end{equation*} 
	Then, $\hbox{the v.w.s. $\omega $ found in Theorem \ref{t2.3} is unique.}$
\end{theo}
This theorem relies on the following general result which does not  require any information about the  boundary condition, since the required additional information is written in another way :
\def\ld{{\rm L}}
\begin{theo}\label{t223}{\bf (Comparison principle)}\ \\
	Let $\OV\omega$ be in $L^1(\O;\d^{-r})\cap W^{1,1}_\loc(\O),\ r>1$. Let $\OV\omega \in L^{N',\infty}(\O)$ and $\ORA u\in L^{p,1}(\O)$    with $p>N$ or $p=N$ with a small norm. Assume that
	\begin{equation*}\ld\OV\omega\dot\equiv-\Delta\OV \omega+\div(\ORA u\,\OV\omega)\LEQ0\hbox{ in }\D'(\O).\end{equation*}
	Then\begin{equation*}\OV\omega\LEQ0\hbox{ in }\O.\end{equation*}
\end{theo}
As an immediate corollary of the above theorem we have
\begin{cor}{\bf of Theorem \ref{t223}}\  \\
	Assume the hypotheses of Theorem \ref{t223} hold and let $f\in L^1_{loc}(\O)$. Then there exists at most one function $\OV\omega\in  L^1(\O;\d^{-r})\cap W^{1,1}_\loc(\O)$, $r>1$ solution of $\ld\OV\omega=f$ in $\D'(\O)$.
\end{cor}
For the proof of Theorem \ref{t223}, we need the following extension of the Kato's inequality whose proof is similar to the one given in \cite{MV} :
\begin{theo}\label{t224}{\bf (Local Kato's inequality)}\ \\
	Let $\OV\omega\in W^{1,1}_{loc}(\O)$ with $\ORA u\  \OV\omega\in L^1_{loc}(\O)$. Assume that $\ld\OV\omega=-\
	\Delta \OV\omega+\div(\ORA u\ \OV\omega)$ belongs to $L^1_{loc}(\O).$ 
	Then \begin{enumerate}
		\item $\forall\,\psi\in\D(\O),\ \psi\GEQ0, \DST\ \int_\O\OV\omega_+ \ld^*\psi dx\LEQ\int_\O\psi\sign_+(\OV\omega)\ld(\OV\omega)dx$,
		\begin{equation*}i.e.\ \ld(\OV\omega_+)\LEQ\sign_+(\OV\omega)\ld(\OV\omega)\ \hin\ \D'(\O).\end{equation*}
		\item $\ld(|\OV\omega|)\LEQ\sign(\OV\omega)\ld(\OV\omega)$ in $\D'(\O)$.\\ Here
		\begin{equation*} \sign_+(\s)=\begin{cases}1&if\ \s>0,\\0&if\ \s\LEQ0,\end{cases}\qquad \sign(\s)=\begin{cases}1&if 
		\s>0,\\-1&if\ \s<0,\end{cases}\end{equation*}
		\begin{equation*}\ld^*\psi=-\Delta\psi-\ORA u\cdot\nabla\psi,\ for \ \psi\in C_c^\infty(\O).\end{equation*}
	\end{enumerate}
\end{theo}
 \begin{proof}
Following \cite{MV}, we first remark that for any $\a\in C^\infty_c(\O),\ \ld(\a\OV\omega)\in L^1(\O)$ since, one has, in $\D'(\O)$,
\begin{equation*}\ld(\a\OV\omega)=\a \ld\OV\omega-\OV\omega\Delta\a-2\nabla\OV\omega\cdot\nabla\a+(\ORA u\,\OV\omega)\cdot\nabla\alpha\in L^1(\O).\end{equation*}

Thus, the conclusion 1. will be proved if we show that
\begin{equation*}\ld(\a\OV\omega)_+\LEQ\sign_+(\a\omega)\ld(\a\OV\omega)\ \hin\ \D'(\O).\end{equation*}
For this purpose, we may assume that $\OV\omega\in W^{1,1}(\O)$ with compact support and $\ld\OV\omega\in L^1(\O)$. Moreover, if $\rho_j\in C^\infty_c(\R^N)$ is a sequence of mollifiers, and $\OV\omega\star\rho_j\in C^\infty_c(\O)$ we have 
\begin{equation*}\ld(\OV\omega\star\rho_j)=\ld\OV\omega\star\rho_j\to \ld\OV\omega\ \hin\ L^1(\O).\end{equation*}
So, it is sufficient to show the inequality number for $\OV\omega\in C^\infty_c(\O)$. From here, we argue as for the case where $\ld$ is replaced by the Laplacian operator (see Proposition 1.5.4 p.21 in \cite{MV} for more details). We   approximate the functions $\sign_+$ by a sequence of convex, non-decreasing functions $h_\eps$ such that
\begin{equation*}\lim_{\eps\to 0} h'_\eps(t)=\sign_+(t);\quad \lim_{\eps\to0}h_\eps(t)=t_+\end{equation*}
\begin{equation*}\sup_{\eps>0}|h'_\eps|(t) \hbox{ is independent of }\eps.\end{equation*}
Thus, for all $\psi\in C^\infty_c(\O),\ \psi\GEQ0$, we have
\begin{equation}\label{eq750}
\int_\O h_\eps(\OV\omega)\ld^*\psi dx\LEQ\int_\O\psi h'_\eps(\OV\omega)\ld\OV\omega dx,
\end{equation}
where $\ld^*\psi=-\Delta\psi-\ORA u\cdot\nabla\psi$.\\

Indeed, $\psi h'_\eps(\OV\omega)$ is  in  $C^\infty_c(\O)$ and then the convexity of $h_\eps$ implies
\begin{equation*}\int_\O\psi h'_\eps(\OV\omega)\ld\OV\omega dx\GEQ-\int_\O h_\eps(\OV\omega)\Delta\psi dx+\int_\O\ORA u\psi h'_\eps(\OV\omega)\cdot  \nabla\OV\omega dx.\end{equation*}
Since $\div(\ORA u)=0$, and $h'_\eps(\OV\omega)\nabla\OV\omega=\nabla h_\eps(\OV\omega)$ we have 
\begin{equation*}\int_\O\ORA u\psi h'_\eps(\OV\omega)\cdot\nabla\OV\omega dx=\int_\O\ORA u\psi\cdot\nabla h_\eps(\OV\omega)dx=
-\int_\O\ORA u\cdot \nabla\psi h_\eps(\OV\omega)dx.\end{equation*}
Thus we get (\ref{eq750}).\\

As in \cite{MV}, letting $\eps\to0$, we have
\begin{equation*}\int_\O\OV\omega_+\ld^*\psi dx\LEQ \int_\O\psi\sign_+(\OV\omega)\ld\OV\omega dx \quad\forall\,\psi\in\D(\O),\ \psi\GEQ0.\end{equation*}
We derive conclusion1., as in \cite{MV}, for $\OV\omega\in W^{1,1}_c(\O)$ and  the same for conclusion 2.
\end{proof}

To extend  the set of test functions from $\D(\O)$ to   other sets of functions we need the following approximation result. 

\begin{lem}\label{l800}{\bf (Approximation of functions in  $W^{m,\infty}(\O)$ by a sequence in $W^{m,\infty}_c(\O)$)}\ \\
	Let $W^{m,\infty}_c(\O)=\Big\{\f\in W^{m,\infty}(\O) \hbox{ with compact support}\Big\}$, $1<m<+\infty$ and assume that $\p\O$ is of class $C^m,\ r>0$. Then, for $\f\in W^{m,\infty}(\O)$ there exists a sequence $(\f_n)_n,\ \f_n\in W^{m,\infty}_c(\O)$, such that
	\begin{enumerate}
		\item $\d^r(D^\a\f_n)\to\d^r(D^\a\f)$ strongly in $L^\infty(\O)$, for all $\a $ such that $|\a|<r$.
		\item Moreover, if $\f\in W^{1,\infty}_0(\O)$ then
		\begin{equation*}\sup_n||\nabla\f_n||_\infty\LEQ c_\O||\nabla\f||_\infty, \hbox{ ($c_\O$ with independent of $\f$)},\end{equation*}
		\begin{equation*}\d^r(D^\a\f_n)\to\d^r(D^\a\f)\hbox{ strongly in $L^\infty(\O)$ for } |\a|<r+1. \end{equation*}
		\item If $\f\GEQ0$ then one can take $\f_n\GEQ0$.
		\item If $\f\in C^m(\OV\O)$ then $\f_n\in C_c^m(\O)$. By the density of $C^\infty_c(\O)$ in $C^m_c(\O),\ \f_n$ in this case can be taken in $C^\infty_c(\O)$.
	\end{enumerate}
\end{lem}
\begin{proof}
Let $h\in C^\infty(\R)$ be such that $0\LEQ h\LEQ1$, $h(\s)=\begin{cases}1&\hbox{if }\s\GEQ1,\\0&\hbox{if\ }\s\LEQ0\end{cases}$\\
Since $\p\O\in C^m$, $\d$ is of class $C^m$ in a neighborhood $U$ of $\p\O$ (see \cite{g1}). Let $0<\eps<1$ be such that
\begin{equation*}\Big\{x\in\O:\d(x)\LEQ\eps\Big\}\subset U\end{equation*}
and define, for $x\in\O$,
\begin{equation}\label{eqet}h_\eps(x)=h\left(\dfrac{2\d(x)-\eps}\eps\right),\end{equation}
so that $h_\eps(x)=1$ if $\d(x)>\eps,$ $ h_\eps(x)\to1$\ as\ $\eps\to0,$ and $h_\eps(x)=0$ if $\d(x)<\eps/2.$\\
One has
\begin{equation*}|D^\a h_\eps(x)|\LEQ c\,\eps^{-|\a|},\hbox{
	for a constant $c>0$ independent of $x$ and $\eps$.}\end{equation*}
Since we have, by Leibniz's formula
\begin{equation}\label{eq801}
D^\a\big(\f(1-h_\eps)\big)(x)=\sum_{\b+\c=\a}c_{\c\b}D^\b\f(x)D^\c(1-h_\eps)(x),
\end{equation}
($c_{\c\b}$ are constant depending only on $\c,\ \b$) and for $\c\neq0$.
\begin{equation}\label{es802}
\d^r(x)\big|D^\c h_\eps(x)\big|\LEQ c\,\eps^{-|\c|+r},
\end{equation}
we then deduce, that
\begin{equation*}\d^r(x)\big|D^\a\big(\f(1-h_\eps)\big)(x)\big|\LEQ c\left[\sum_{\b+\c=\a,\,\c\neq0}|D^\b\f(x)|\eps^{-|\c|+r}+\d^r|D^\a\f|(1-h_\eps)\right].\end{equation*}
Therefore
\begin{equation}\label{eq803}
\sup_{x\in\O}\d^r(x)\big|D^\a\f(1-h_\eps)(x)\big|\LEQ c\,\eps^{-|\a|+r}.
\end{equation}
Taking $\eps=\dfrac1n,$ and $\f_n=h_{\frac1n}\f$ is convenient for large $n\GEQ n_0$.
If furthermore $\f\in W^{1,\infty}_0(\O)$ then
\begin{equation*}|\f(x)|\LEQ\d(x)||\nabla\f||_\infty.\end{equation*}
Hence,
\begin{equation*}\d^r\big|D^\a\big(\f(1-h_\eps)\Big)(x)\big|
\LEQ c\,\d^{r+1}(x)\eps^{-|\a|}+c\!\!\!\!\!\!\sum_{\b\neq0,\b+\c=\a}\!\!\!\!\!\!\!\big|D^\b\f|(x)\big|D^\c(1-h_\eps)|\d^r(x)\LEQ c\,\eps^{-|\a|+r+1}.\end{equation*}
On the other hand
\begin{equation*}\begin{cases}
\hbox{on\ }\d(x)\LEQ\eps&\big|\nabla(\f h_\eps)(x)\big|\LEQ|\f(x)|\nabla h_\eps(x)|+2||\nabla\f||_\infty\LEQ c\,||\nabla\f||_\infty\left[1+\dfrac{\d(x)}\eps\right]\\&\ \qquad\quad\qquad\LEQ c\,||\nabla\f||_\infty,\\
\hbox{on\ }\d(x)>\eps&\big|\nabla\f_n(x)\big|\LEQ2||\nabla\f||_\infty.
\end{cases}\end{equation*}
Moreover, one has
\begin{equation*}\d^r(x)\Big|\nabla\big(\f(1-h_\eps)\big)\Big|(x)
\LEQ\d^r|D\f|(x)\big(1-h_\eps(x)\big)+c\,\d^{r+1}(x)||\nabla\f||_\infty|\nabla h_\eps|\LEQ c\,\eps^r.\end{equation*}
\  \end{proof}

Thanks to the above approximation lemma we can modify the set of the test functions in the Kato's inequality as follows
\begin{cor} {\bf (of Theorem \ref{t224} : Variant of Kato's inequality)}\label{c4t224}\ \\
	Let $\OV\omega$ be in $W^{1,1}_\loc(\O)\cap L^{N',\infty}(\O)$, $\OV\omega\in L^1(\O;\d^{-r})$ for $r>1$ and $\ORA u\in  L^{N,1}(\O)^N$ with $\div(\ORA u)=0,\ \ORA u\cdot\ORA n=0$. Assume furthermore that 
	$\ld\OV\omega=-\Delta\OV\omega+\div(\ORA u \OV\omega)$ is in $L^1(\O;\d)$. \\Then for all $\phi\in C^2(\OV\O),\ \phi=0$ on $\p\O$, $\phi\GEQ0$ one has
	\begin{enumerate}
		\item $\DST\int_\O\OV\omega_+\ld^*\phi dx\LEQ\int_\O\phi\,\sign_+(\OV\omega)\ld(\OV\omega)dx,$
		\item $\DST\int_\O|\omega|\ld^*\phi dx\LEQ\int_\O\phi\,\sign(\OV\omega)\ld(\OV\omega)dx,$
	\end{enumerate}
	where $\ld^*\phi=-\Delta\phi-\ORA u\cdot\nabla \phi=-\Delta\phi-\div(\ORA u\phi).$
\end{cor}
\begin{proof}
Let $\phi\GEQ0$ be in $C^2(\OV\O)$ with $\phi=0$ on $\p\O$. Then according to Lemma \ref{l800}, we have a sequence $\phi_n\in C^2_c(\O)$, $\phi\GEQ0$, such that
\begin{equation*}\begin{cases}
\d^r\Delta\phi_n\to\d^r\Delta\phi &\hin\ \ C(\OV\O)\hbox{ for }r>1,\\
\d^r\nabla\phi_n\to\d^r\nabla\phi &\hin\ \ C(\OV\O)^N,\ ||\nabla\phi_n||_\infty\LEQ c||\nabla\phi||_\infty.
\end{cases}\end{equation*}
Therefore
\begin{equation*}\lim_{n\to+\infty}\int_\O\OV\omega_+\Delta\phi_ndx=\lim_{n\to+\infty}\int_\O\OV\omega_+\d^{-r}\cdot\d^r\Delta\phi_ndx=\int_\O\OV\omega_+\Delta\phi dx,\end{equation*}
since $\OV\omega_+\in L^1(\O;\d^{-r})$ and $r>1$.\\
By the Lebesgue dominated convergence theorem, one has
\begin{equation*}\lim_{n\to+\infty}\int_\O\ORA u\cdot\nabla\phi_n\OV\omega_+dx=\int_\O\ORA u\cdot\nabla \phi\OV\omega_+dx,\hbox{ since } \ORA u\cdot\OV \omega_+\in L^1(\O)^N.\end{equation*}
Therefore
\begin{eqnarray*}
	\int_\O\OV\omega_+\ld^*\phi dx&=&\lim_{n\to\infty}\int_\O\OV\omega_+\ld^*\phi_ndx\LEQ\lim_{n\to+\infty}\int_\O\phi_n\sign_+\OV\omega\,\sign_+\ld\OV\omega dx\\
	&=&\int_\O\phi\,\sign_+\OV\omega\ld\OV\omega \hbox{ (since }\dfrac{|\phi_n|}\delta\LEQ ||\nabla\phi_n||_\infty\LEQ c||\nabla\phi||_\infty).
\end{eqnarray*}

Now we come to the proof of the uniqueness result stated in Theorem \ref{t222}.\\
\begin{proof}[Proof of Theorem \ref{t222}]
Since the v.w.s. $\omega$ satisfies $V\omega\in L^1(\O;\d)$, so if $V\GEQ c\,\d^{-r}$, for $r>2$,  we have in a neighborhood $U$ of $\p\O$
\begin{equation*}\int_\O|\omega|\d^{- (r-1)}dx\LEQ c\int_U V|\omega|\d dx+c_1\int_\O|\omega|dx<+\infty.\end{equation*}
Thus $\omega\in L^1(\O;\d^{-\widetilde r})$ with $\widetilde r=r-1>1$ for $r>2$.\\ If $\omega_1,\ \omega_2$ are two v.w.s. then $\omega=\omega_1-\omega_2$
\begin{equation*}\ld\omega=\ld(\omega_1-\omega_2)=-\Delta\omega+\div(\ORA u \omega)=-V\omega\in L^1(\O;\d).\end{equation*}
We deduce from the  Corollary \ref{c4t224}  of Theorem \ref{t224} that $\forall\,\phi\GEQ0,\ \phi\in C^2(\OV\O),\ \phi=0$ on $\p\O$
\begin{equation*}\int_\O|\omega|\ld^*\phi dx\LEQ-\int_\O\phi\,\sign(\omega)V\omega dx=-\int_\O\phi\, V|\omega|dx\LEQ0.\end{equation*}
For $\ORA u\in L^{p,1}(\O)^N$, ($p\GEQ N$ as in the statement of Theorem~\ref{t222}) let us consider $\phi_0\in H^1_0(\O)$ solution of 
\begin{equation*}\ld^*\phi_0=-\Delta\phi_0-\ORA u\,\nabla\phi_0=1.\end{equation*}
Then $\phi_0\GEQ0,\ \phi_0\in W^2 L^{p,1}(\O)$ according to the above regularity result, (see Propositions \ref{p8} or \ref{p9}) and $\phi_0$ can be approximated by a sequence $\phi_{0j}\in C^2(\OV\O),\ \phi_{0j}\GEQ0$,\ $\phi_{0j}=0$ on $\p\O$ satisfying
\begin{equation*}\ld_j^*\phi_{0j}=-\Delta\phi_{0j}-\ORA {u_j}\cdot\nabla\phi_{0j}=1,\ {\ORA u}_j\to\ORA u\hbox{ in }L^{p,1},\ {\ORA u}_j\in\V,\end{equation*}
so that
\begin{equation*}||\phi_{0j}||_{W^2L^{p,1}}\LEQ c.\end{equation*}
Indeed, we may assume that $\phi_{0j}$ converges weakly to a function $\OV\phi_0$ in $W^2L^{p,1}(\O)$, \begin{equation*}\nabla\phi_{0j}(x) \to\nabla\OV\phi_0(x)\hbox{ and }\phi_{0j}(x)\to\OV\phi_0(x)\ a.e.\ x\in\O.\end{equation*}
Since
\begin{equation*}\int_\O|\omega|\,|\ORA {u_j}-\ORA u|\LEQ||\ORA {u_j}-\ORA u||_{L^{N,1}}|\omega|_{L^{N',\infty}},\end{equation*}
and \begin{equation*}||\nabla\phi_{0j}||_\infty\LEQ c,\end{equation*}
we deduce that
\begin{equation*}\lim_{j\to+\infty}\int_\O|\omega|{\ORA u}_j\cdot\nabla\phi_{0j}=\int_\O|\omega|\ORA u\cdot\nabla\OV\phi_0 dx.\end{equation*}
Thus 
\begin{equation*}\ld^*\OV\phi_0=1,\quad\OV\phi_0\in W^2L^{N,1}(\O)\cap H^1_0(\O).\end{equation*}
By uniqueness $\OV\phi_0=\phi_0$ and then $\DST \ld_j^*\phi_{0j}\rightharpoonup \ld^*\phi_0\hbox{ weakly in }L^{N,1}.$
Since, we have 
\begin{equation*}\int_\O|\omega|\ld^*\phi_{0j} dx\LEQ0.\end{equation*}
\begin{equation*}0\LEQ\int_\O|\omega|dx=\int_\O|\omega|\ld_j^*\phi_{0j}dx\LEQ\int_\O|\omega|(\ld^*_j\phi_{0j}-\ld^*\phi_{0j})dx\xrightarrow[j\to+\infty]{}0,\end{equation*}
we arrive to $\omega=0$. \end{proof}

\begin{rem}\ \\
 In Theorem \ref{t223} and Theorem \ref{t224}, if $\ORA u\equiv 0$ (or $\ORA u\in C^1(\OV\O)^N$) then we can weaken  the conditions on $\OV \omega$ reducing it to $\OV\omega$ belongs to  $L^1(\O;\d^{-r}),\ r>1$. Then the above conclusions hold true.
\end{rem}
\begin{rem}\ \\
	In fact, in Corollary 3, we can state that the unique solution of (\ref{eq1}) (without any indication of the boundary condition) {\bf must} satisfy that $\omega=0$ on $\p\O$ at least if $\omega$ is differentiable. Indeed, a consequence of Lemma 7 we have 
	\begin{equation*}L^1(\O,\d^{-r})\cap W^1L^{p,q}(\O)=W_0^1L^{p,q}(\O)\hbox{ if }r>1 \ (1\LEQ p,q\LEQ+\infty).\end{equation*}
\end{rem}

\begin{rem} There is a large amount of works in the literature in which the uniqueness of solutions of suitable elliptic problems is established without indicating any boundary condition but these previous papers deal with degenerate elliptic operators (see, e.g. \cite{BaouendiMS}, \cite{BG}, \cite{DT} and the references therein). We point out that the main reason to get this type of results in our case (in which the diffusion operator is the simplest one and is not degenerate) is the presence of a very singular coefficient of the zero order term (the potential $V(x)$) which is "pathological" since it is more singular on the boundary of the domain than what the Hardy inequality may allow.
\end{rem}

\subsection{Boundedness in $L^{N'}(\O)$ of the v.w.s., regularity and blow-up in absence of any  potential ($V=0$)}\label{s4}

Since the very weak solutions found in Theorem \ref{t2.3} needs not be in $L^1(\O)$ our main goal now (assuming $V\equiv0$) is to analyze under which conditions $\omega$ is globally integrable. We have
\begin{theo}\label{t1}{\bf (Integrability in $L^{N'}(\O)$.)}\\
	Let $f$ be in $L^1\Big(\O;\d\big(1+|\ln\d|\big)^{\frac1{N'}}\Big),\ \dfrac1N+\dfrac1{N'}=1$, $V=0$,   $\ORA u\in (L^N(\ln L)^{\frac\b N})^N,$ with $\b>N-1$, $\div(\ORA u)=0$ in $\O$ and $\ORA u\cdot\ORA n=0$ on $\p\O$. Then the unique very weak solution $\omega$ of   equation (\ref{eq1}) belongs to $L^{N'}(\O)$.
\end{theo}
We recall the
\begin{lem}\label{l20}{\bf (see \cite{RakoJFA})}\ \\
	Let $\O$ be a bounded open Lipschitz set and $\a>0$. Then, there exists a constant $c_\a(\O)>0$ such that $\forall\,\phi\in W^1_0L^\a_\exp(\O)$
	\begin{equation*}|\phi(x)|\LEQ c_\a(\O)\d(x)(1+|\ln\d(x)|)^\a||\nabla \phi||_{L^\a_{\exp}(\O)}.\end{equation*}
\end{lem}

\begin{proof}[Proof of Theorem \ref{t1} (boundedness in $L^{N'}(\O)$)]
Let $\omega$ be the very weak solution found in Theorem \ref{t2.3}  and assume that 
\begin{equation*}f\in L^1\big(\O;\d(1+|\ln \d|)^{\frac1{N'}}\big).\end{equation*} 
We know that there exists a sequence $\ORA u_j\in\V$ such that the corresponding sequence $(\omega_j)_j$ satisfying relation (\ref{eq18}) verifies $\omega_j\rightharpoonup \omega$ weak-* in $L^{N',\infty}$ and that $\forall\,\phi\in H^1_0\cap W^2L^N(\O)$
\begin{equation}\label{eq27}
\int_\O\omega_j\big[-\Delta\phi-\ORA{ u_j}\nabla\phi\big]dx=\int_\O f\phi dx.
\end{equation}
Here $\ORA u_j$ converges in $(L^N(\ln L)^{\frac\b N})^N=\Lambda$ to $\ORA u$ strongly where  $\b>N-1$. Let $g\in L^N(\O)$and let $\phi_j$ be the solution of 
\begin{equation*}\phi_j\in W^2L^N(\O)\hbox{ such that }
-\Delta\phi_j-{\ORA u}_j\nabla \phi_j=g\hbox{ in }\O,\ \phi_j=0\hbox{ on }\p\O.\end{equation*}
Then according to Theorem \ref{t4}, we have 
\begin{equation*}||\phi_j||_{W^2L^N(\O)}\LEQ K_\eps\dfrac{1+||{\ORA u}_j||_\Lambda}{1-\eps||{\ORA u}_j||_\Lambda}||g||_{L^N(\O)},
\end{equation*}
with
\begin{equation*}\eps\sup_j||{\ORA u}_j||_\Lambda\LEQ\dfrac12,\hbox{ for some $\eps>0$.}\end{equation*}
Thus
\begin{equation}\label{eq28}
||\phi_j||_{W^2 L^N(\O)}\LEQ K(\O)||g||_{L^N(\O)}.
\end{equation}
By the Trudinger's type inclusion (see Lemma \ref{l})
\begin{equation}\label{eq29}
||\nabla \phi_j||_{L^{\frac1{N'}}_\exp}\LEQ K_{10}||\phi_j||_{W^2L^N(\O)}\LEQ K_{11}||g||_{L^N(\O)}.
\end{equation}
Therefore, considering equation (\ref{eq27}), we have
\begin{equation}\label{eq30}
\int_\O\omega_jgdx=\int_\O f\phi_jdx,
\end{equation}
with the help of Lemma \ref{l20}  with $\a=\dfrac1{N'}$ and estimate (\ref{eq29}), this relation gives:
\begin{equation}\label{eq31}
\int_\O\omega_jg dx\LEQ K_{12}||g||_{L^N}\int_\O|f|\d(x)(1+|\ln\d(x)|)^{\frac1{N'}}dx.
\end{equation}
Hence
\begin{equation}\label{eq32}
\sup_{||g||_{L^N}=1}\int_\O\omega_j gdx\LEQ K_{12}\int_\O|f|\d(x)(1+|\ln\d(x)|)^{\frac1{N'}}dx,
\end{equation}
which shows that :
\begin{equation}\label{eq33}
||\omega||_{L^{N'}(\O)}\LEQ K_{12}\int_\O|f|\d(x)(1+|\ln\d(x)|)^{\frac1{N'}}dx,
\end{equation}
proving the result.  \end{proof}

For the case $V\equiv 0$, we can always obtain the $W^{1,q}(\O)$-regularity, for $q\GEQ1$, provided some integrability on $f$ but also on $\ORA u$. Here is a first result in that direction~:

\begin{theo}\label{t2.4}\ \\
	Let $f$ be in $L^1(\O;\d(1+|\ln\d|))$, $V=0$, and $\ORA u$ in $\bmor(\O)^N$. Then, the very weak solution found in Theorem \ref{t2.3}  belongs to $W^{1,1}_0(\O).$
\end{theo}
 \begin{proof}
As before we consider the approximating problem (\ref{eq17}) with ${\ORA u}_j=\ORA u$, say
\begin{equation*}\begin{cases}
-\Delta \omega_j+\ORA u\cdot \nabla \omega_j=f_j &\hin\ \ \O,\\
\omega_j\in H^1_0(\O)\cap W^2L^{p,1}(\O)&\forall\,p<+\infty.
\end{cases}
\end{equation*}
Thus, taking $\phi\in W^1_0\bmor(\O)$ we have
\begin{equation*}\int_\O\nabla\omega_j\cdot\nabla\phi\, dx
+\int_\O\ORA u\cdot\nabla\omega_j\phi\, dx
=\int_\O f_j\phi dx\Longleftrightarrow 
\int_\O\big[\nabla\omega_j\cdot\nabla\phi-\ORA{u}\cdot\nabla \phi\,\omega_j\big]dx=\int_\O f_j\phi\, dx.\end{equation*}
Let $F_j= 
\dfrac{\nabla \omega_j}{|\nabla \omega_j|}\hbox{ if }\ \nabla\omega_j\neq0,\hbox{ and } 
0\hbox{ otherwise,}\quad F_j\in L^\infty(\O)^N,\quad ||F_j||_\infty\LEQ1.$ According to Proposition \ref{p7}, there exists a function $\phi_j\in W^1_0\bmor(\O)$ such that 
\begin{equation*}-\Delta\phi_j-\ORA u\nabla\phi_j=-\div(F_j),\hbox{ and }||\phi_j||_{W^1_0L^q}\LEQ c_9||F_j||_{L^q}\LEQ c_q<+\infty\ \forall\,q>1,\end{equation*}\begin{equation*}
\Longleftrightarrow\int_\O\nabla\phi_j\nabla\f\,dx-\int_\O\ORA u\nabla\phi_j\f \,dx=\int_\O F_j\nabla \f\,dx\quad\forall\,\f\in H^1_0(\O).\end{equation*}
Choosing $\f=\omega_j$, we have
\begin{equation}\label{eq34}
\int_\O|\nabla\omega_j| dx=\int_\O\nabla\phi_j\cdot\nabla\omega_j\,dx-\int_\O\ORA u\cdot\nabla\phi_j\omega_j\,dx=\int_\O f_j\phi_j\,dx.
\end{equation}
From\,Lemma\,\ref{l20},   and  by the\,John-Nirenberg\,inequality\,(see \cite{Tor}) we have~:
\begin{equation}\label{eq35}
|\phi_j(x)|\LEQ c(\O)\d(x)(1+|\ln\d(x)|)||\nabla \phi_j||_{L_\exp}
\LEQ c(\O)\d(x)(1+|\ln\d(x)|)||\nabla\phi_j||_{\bmor(\O)}.
\end{equation}
We recall that
\begin{equation}\label{eq36}
||\nabla \phi_j||_\bmor\LEQ K(||F_j||_\infty+||\ORA u\phi_j||_\bmor)\LEQ c,
\end{equation}
since $\phi_j\to\phi$ strongly in $C^{0,\a}(\OV\O)$ (see Proposition \ref{p7}).\\
Combining (\ref{eq34}) to (\ref{eq36}), we have
\begin{equation}\label{eq37}
\int_\O|\nabla\omega_j|dx\LEQ c\int_\O|f_j|\d(x)(1+|\ln\d|)dx\LEQ K\int_\O|f|\d(1+\ln\d|)dx;
\end{equation}
using also the fact that
\begin{equation*}\begin{cases}
\omega_j\rightarrow \omega\hbox{\hbox{ strongly } in }L^q(\O)\ q<N',\\
\omega_j\rightharpoonup\omega\hbox{ weakly in }W^{1,q}_\loc(\O)\ 1<q<1+\dfrac1N,
\end{cases}\end{equation*}
we deduce that :
\begin{equation*}\int_\O|\nabla\omega|dx\LEQ c\int_\O|f|\d(1+|\ln\d|)dx.\end{equation*}
\  \end{proof}

Let us prove that if we enhance the integrability condition on  $f$ to $f\in L^1(\O,\d^\a)$ for some $\a\in]\,0,1\,[$ then we can weaken the condition on $\ORA u$ to $\ORA u\in L^{\frac N{1-\a}}(\O)^N$ and in that case we have
\begin{theo}\label{t2.5}\ \\
	Let $f$ be in $L^1(\O,\d^\a)$ for some $\a\in]\,0,1\,[$, $V=0$, $\ORA u\in L^{\frac N{1-\a}}(\O)$ with $\div(\ORA u)=0,$ $\ORA u\cdot \ORA n=0$ on $\p\O$. Then, the very weak solution $\omega$ found in Theorem \ref{t2.3}   belongs to $W^1_0\LNua$. Moreover, there exists a constant $K(\a;\O)>0$ such that
	\begin{equation*}||\omega||_{W^1_0\LNua}\LEQ K(\a;\O)\big(1+||\ORA u||_{L^{\frac N{1-\a}}} \big)||f||_{L^1(\O,\d^\a)}.\end{equation*}
\end{theo}
The proof of Theorem \ref{t2.5} relies on the following result, dual  of Proposition~\ref{p5}.
\begin{propo}\label{p15}\ \\
	Let $\ORA u\in L^{p,q}(\O),\ p>N, q\in[\,1,+\infty]$, $V=0$, and $F\in L^{p',q'}(\O)^N,\ \dfrac1p+\dfrac1{p'}=1=\dfrac1q+\dfrac1{q'}$. Then there exists $\OV\omega\in W^1_0L^{p',q'}(\O)$ such that
	\begin{equation}\label{eq37a}
	-\Delta\OV\omega+\ORA u\cdot\nabla\OV\omega=-\div(F),\end{equation}
	\hbox{which is equivalent to }
	\begin{equation}\label{eq37aa}
	a(\OV\omega;\phi)=\int_\O\nabla\OV\omega\cdot\nabla\phi\,dx+\int_\O\ORA u\cdot\nabla\OV\omega\phi\,dx=\int_\O F\cdot\nabla \phi\,dx
	\end{equation}
	$\forall\,\phi\in W^1_0L^{p,q}(\O)$. Moreover
	\begin{equation*}||\nabla \omega||_{L^{p',q'}}\LEQ K_{pq}(1+||\ORA u||_{L^{p,q}})||F||_{L^{p',q'}}\end{equation*}
\end{propo}
 \begin{proof}
Let $G$ be in $L^{p,q}(\O)^N,\ p>N$. Following Proposition \ref{p5}, there exists   a function $\phi_0\in W^1_0L^{p,q}(\O)$ such that
\begin{equation*}\int_\O\nabla\phi_0\cdot\nabla\f\,fx-\int_\O\ORA u\cdot\nabla\phi_0\f\,dx=\int_\O G\cdot\nabla\f\,dx\qquad\forall\,\f\in C^\infty_c(\O).\end{equation*}
Since
\begin{equation*}-\int_\O\ORA u\cdot\nabla\phi_0\f\,dx=\int_\O\ORA u\cdot\nabla\f\phi_0\,dx,\end{equation*}
by using a density argument over the set of test functions there exists
\begin{equation}\label{eq38}
a(\f,\phi_0)=\int_\O\nabla\phi_0\cdot\nabla\f+\int_\O\ORA u\cdot\nabla\f\phi_0=\int_\O G\cdot\nabla \f \,dx\qquad\forall\,\f\in W^1_0L^{p',q'}(\O).
\end{equation}
Let $F_k\in L^\infty(\O)^N$, with $|F_k(x)|\LEQ |F(x)|$ in $\O$. Then we have that $\OV\omega_k\in W^1_0L^{p',q'}(\O)\cap H^1(\O)$ such that
\begin{equation}\label{eq39}
a(\OV \omega_k,\phi)=\int_\O F_k\cdot\nabla\phi\,dx\qquad\forall\phi\in W^1_0L^{p,q}(\O).
\end{equation}
Choosing $\phi=\phi_0$ in this last equation,  we find that 
\begin{equation}\label{eq40}
\int_\O G\cdot\nabla\OV\omega_k\,dx=a(\OV\omega_k,\phi_0)=\int_\O F_k\cdot\nabla\phi_0\,dx.
\end{equation}
Following Proposition \ref{p5}, we have
\begin{equation}\label{eq41}
||\nabla\phi_0||_{L^{p,q}}\LEQ K_{pq}(1+||\ORA u||_{L^{p,q}})||G||_{L^{p,q}}.
\end{equation}
From relation (\ref{eq40}) and (\ref{eq41}), we have
\begin{equation}\label{eq42}
\int_\O G\cdot\nabla\OV\omega_k\,dx\LEQ K_{pq}(1+||\ORA u||_{L^{p,q}})||F_k||_{L^{p',q'}}||G||_{L^{p,q}}.
\end{equation}
So that we have
\begin{equation}\label{eq43}
\sup_{||G||_{L^{p,q}}=1}\int_\O G\cdot\nabla\OV \omega_k dx\LEQ K_{pq}(1+||\ORA u||_{L^{p,q}})||F||_{L^{p',q'}}
\end{equation}

\begin{equation}\label{eq44}
||\nabla\OV{\omega_k}||_{L^{p',q'}}\LEQ K_{pq}(1+||\ORA u||_{L^{p,q}})||F||_{L^{p',q'}}.
\end{equation}
By  standard argument, we derive the existence of $\OV\omega$ satisfying (\ref{eq37a}) as a weak limit of $\OV\omega_k$ in $W^1L^{p',q'}(\O)$. \end{proof}

\begin{proof}[Proof of Theorem \ref{t2.5}]
Since $f\in L^1(\O;\d^\a)$, according to   \cite{DRJFA}, there exists $F=\nabla v\in L^{\frac N{N-1+\a}}(\O)^N$, $f=-\div(F)$. Moreover, the function $F_k=\nabla v_k$ satisfying $-\Delta v_k=T_k(f)$ converge to $F$ strongly in $L^{\frac N{N-1+\a}}(\O)^N$ ($v_k$ and $v$ are in $W^1_0\LNua$).

Since the very weak solution $\omega$ found in Theorem \ref{t2.3}   is the weak-* limit of the solutions of the regularized problem
\begin{equation*}\begin{cases}
-\Delta \omega_k+\ORA u\cdot\nabla \omega_k=f_k=T_k(f)=-\div(F_k),\\
\omega_k\in W^2L^q(\O)\cap H^1_0(\O)\hbox{ with } q=\dfrac N{1-\a}>N,
\end{cases}\end{equation*}
and \begin{equation*}||\nabla\omega_k||_{L^{q'}(\O)}\LEQ K_q(1+||\ORA u||_{L^q})||F_k||_{L^{q'}(\O)},\quad q'=\dfrac N{N-1+\a},\end{equation*}
letting $k\to+\infty$, we derive the result once we know that $||F||_{L^{q'}}\LEQ c||f||_{L^1(\O,\d^\a)}$. \end{proof}

When $\a=0$, that is $f\in L^1(\O)$, we can weaken the integrability assumption on $\ORA u$ as we state in the following result :
\begin{theo}\label{t10.1xx}\ \\
	Let $f$ be in $L^1(\O)$, $V=0$, $\ORA u\in L^N(\O)^N$ with $\div(\ORA u)=0$ on $\p\O$, $\ORA u\cdot\ORA n=0$ on $\p\O$. Then, the very weak solution $\omega $ found in Theorem \ref{t2.3}  belongs to $W^1_0L^{N',\infty}(\O)$.\\
	Moreover, there exists a constant $c(\O)>0$, independent of $\ORA u$, such that 
	\begin{equation*}||\nabla \omega||_{L^{N',\infty}(\O)}\LEQ c(\O)||f||_{L^1(\O)}.\end{equation*}
\end{theo}
 \begin{proof}
Let $\ORA u_j\in{\mathcal V}$ be such that $\ORA u_j\to \ORA u$ in $L^N(\O)^N$,   and let $f_j\in L^\infty(\O)$ be such that $|f_j(x)\LEQ|f(x)|$ and $f_j(x)\to f(x)$ a.e, $x\in\O.$\\
Let us consider the functions $\omega_j\in W^2L^m(\O)\cap H^1_0(\O)$ $\forall\, m<+\infty$ satisfying
\begin{equation*}-\Delta\omega_j+\ORA u_j\cdot\nabla\omega_j=f_j.\end{equation*}
Then
\begin{equation*}\int_\O|\nabla T_k(\omega_j)|^2dx+\int_\O\ORA u_j\cdot\nabla\int_0^{\omega_j}T_k(\s)d\s=\int_\O T_k(\omega_j)f_j(x)dx,\end{equation*}
and  since by integration by parts we have 
$\DST\int_\O\ORA u_j\cdot\nabla\int_0^{\omega_j}T_k(\s)d\s=0$ we get
\begin{equation}\label{eq1001xx}
\int_\O|\nabla T_k(\omega_j)|^2dx\LEQ k\int_\O|f(x)|dx.
\end{equation}
By the Poincar\'e-Sobolev inequality
\begin{equation*}\int_\O| T_k(\omega_j)|^2dx\LEQ c_\O k\int_\O|f(x)|dx.\end{equation*}
By Proposition \ref{p11}, we deduce that
\begin{equation*}\|\nabla \omega_j||_{L^{N',\infty}(\O)}\LEQ c_\O\int_\O|f(x)|dx.\end{equation*}
Since $\ORA u_j\to\ORA u$ in $L^N(\O)^N$ and by compactness  $\omega_j\to\omega $ in $L^{N'}(\O)$\\ \Big(note that $W^1L^{N',\infty}(\O)\hookrightarrow L^{\frac N{N-2},\infty}(\O)$ for $N\GEQ3$ ( see \cite{RakoRR})\Big), we then have for all $\phi\in C^2(\OV\O)$ with $\phi=0$ on $\p\O$,
\begin{equation*}\int_\O\omega_j \ORA u_j\nabla\phi dx\xrightarrow[j\to+\infty]{}\int_\O\omega\ORA u\cdot\nabla\phi dx,\end{equation*} 
so that  $\omega$ solves (\ref{eq1XV}) for $V\equiv0$. \end{proof}

As for the case $\ORA u=0$, the additional regularity  questions are numerous; for instance, does there exists a datum $f\in L^1(\O;\d)$ for which we have

\begin{equation*}\int_\O|\nabla\omega|dx=+\infty\hbox{ or }\int_\O|\omega|^{N'}dx=+\infty?\end{equation*}

For the explosion of the norm of $\omega$ in $L^{N'}$, we can adopt the same proof as for the explosion of the gradient in $L^1(\O)$. We have
\begin{theo}\label{t6}{\bf (blow-up in $L^{N'}(\O)$)}\\
	Assume that $N\GEQ3,\ \ORA u\in C^{0,\a}(\OV\O)^N,\ \a>0$, $V=0$. Then there exists  a function $f$ in $L^1_+(\O;\d)\backslash L^1(\O,\d(1+|\ln\d|)^{\frac1{N'}})$ such that the very weak solution $\omega$ found in Theorem~\ref{t2.3} satisfies that $\omega$ does not belong to  $L^{N'}(\O)).$
\end{theo}
First we recall the following result that can be proved as in\,\cite{RaAMO}\,(see\,also\,\cite{RabookLinear}).
\begin{lem}\label{l100}
	Let $N\GEQ3$. There exists a function $g\in L^N_+(\O)$ such that the unique solution $\psi\in W^2L^N(\O)\cap H^1_0(\O)$ of $-\Delta\psi-\ORA u\cdot\nabla\psi=g$ satisfies :
	\begin{enumerate}
		\item $\psi(x)\GEQ c_1\d(x),\ \forall\,x$,
		\item $\DST\sup\left\{\dfrac{\psi(x)}{\d(x)},\ x\in\O\right\}=+\infty,$
		\item $L^1_+(\O;\d)\backslash L^1(\O,\psi)$ is non empty.
	\end{enumerate}
\end{lem}
Arguing as in   \cite{RaAMO},   \cite{AbergelRako}, we consider $g_k=T_k(g)$, $g$ given by Lemma \ref{l100} such  that \begin{equation*}\hbox{$\psi_k\in W^2L^q(\O)\cap H^1_0(\O)$ for all $q<+\infty$, $-\Delta\psi_k-\ORA u\nabla\phi_k=T_k(g)$.}\end{equation*}

Now assume   that for all $f\in L^1(\O;\d)$, we have for  the v.w.s. $||\omega||_{L^{N'}}<+\infty$. Then by the Banach-Steinhauss uniform boundedness theorem  as in \cite{AbergelRako, RaAMO}, we derive the existence of a constant $c_0>0$ such that
\begin{equation*}||\omega||_{L^{N'}}\LEQ c_0\int_\O|f|\d dx\quad\forall\,f\in L^1(\O;\d),\end{equation*}
and 
\begin{equation*}\int_\O\omega\big[-\Delta\phi-\ORA u\cdot\nabla\phi\,\big]\,dx=\int_\O f\phi\,dx\qquad\forall\,\phi\in W^2L^{N,1}(\O)\cap H^1_0(\O).\end{equation*}

Taking $\phi=\psi_k$, and $f\in L^1_+(\O;\d)\backslash L^1_+(\O,\psi)$ we see that
\begin{equation}\label{eq100}
0\LEQ \int_\O f\psi_k=\int_\O\omega g_k\,dx\LEQ||\omega||_{L^{N'}}||g||_{L^N}<+\infty.
\end{equation}
Letting $k\to+\infty$, we have a contradiction since
\begin{equation*}\underline{\lim}_{k\to+\infty}\int_\O f\psi_k\GEQ\int_\O f\psi\,dx=+\infty,\end{equation*}
which concludes the proof Theorem \ref{t6}. \end{proof}
\begin{rem}\ \\
	We can give the more precise information that the function $f$ in Theorem \ref{t6}    is not in $L^1(\O;\d(1+|\ln\d|)^{\frac1{N'}})$ (due to Theorem \ref{t1}).
\end{rem}
\subsection{Some final conclusion}
In the opinion of the authors, the results of this paper  open many different further applications in different directions.
Besides the consideration of the list of concrete problems mentioned in the Introduction other studies can be carried out. For instance, following the arguments of   \cite{DiRa2}, 
it is not complicated to extend many of the results of this paper to the study of semilinear problems for which equation (\ref{eq1}) is replaced by the equation
\begin{equation*}-\Delta\omega+\ORA u\cdot\nabla \omega+V\omega+\b(x,u,\nabla u)=f(x)\hbox{ on }\O,\end{equation*}
when $\b$ is nondecreasing in $u$. Moreover the consideration of parabolic problems of the type
\begin{equation*}\omega_t-\Delta\omega+\ORA u\cdot\nabla\omega+V\omega+\b(x,u,\nabla u)=f(t,x)\ on\ \O\times(0,T),\end{equation*}
can be carried out with the help of the results of this paper (mainly the $L^1(\O;\d)$-accretiveness property of the associates operator). The details will be given in some separate work by the authors.\\

\section*{Acknowledgments}
The research of D. G\'omez-Castro was supported by a FPU fellowship from the Spanish government. The research of J.I. D\'iaz and D. G\'omez-Castro was partially supported by the project ref. MTM 2014-57113-P of the DGISPI (Spain). Roger Temam was partially supported by NSF grant DMS 1510249 and by the Research Fund of Indiana University.\\

{\it 
	After this article was completed, we learned, during a presentation at a conference (March 29-30, 2017) in Poitiers, France, that L. Orsina and A. Ponce have obtained related results in the references  \cite{OP1, OP2}. Their results deal essentially with the existence and the use of the normal derivative for any  function  in $W^{1,1}_0(\O)$. In the improved version \cite{OP2} that they  sent to us by the authors after the conference, they add a new proposition (Proposition 2.7) which provides a complement to our results since it gives a qualitative property for $\omega$ solution of our problem (16) if the velocity $\ORA u$ is zero when the solution  
	is integrable on the whole domain (for a right hand side $f$ in $L^1(\O , \d(1 + | Log\d |))$. We note also that J.I.D\'iaz has already derived results similar to their  Proposition 2.7 in \cite{Dinter, 12b}.}

\end{document}